\documentclass[12pt,a4paper]{article}
\usepackage{amssymb,amsfonts,amsmath,amsthm,amscd,array,hhline}
\newtheorem{theorem}{Theorem}
\newtheorem{lemma}{Lemma}[section]
\newtheorem{corollary}{Corollary}[theorem]

\newtheorem{proposition}{Proposition}[section]
\newtheorem{definitionsec}{Definition}[section]
\newtheorem{remark}{Remark}[section]
\DeclareMathOperator{\vr}{\mathrm Var}

\DeclareMathOperator{\spb}{\mathrm{sp}_b}
\DeclareMathOperator{\Lat}{\mathcal L}
\DeclareMathOperator{\chr}{\mathrm char}

\newcommand{\bsp}[1]{\spb\left(#1\right)}
\newcommand{\bspp}[1]{\spb\bigl(#1\bigr)}
\newcommand{\lt}[1]{\Lat\left(#1\right)}
\newcommand{\ltt}[1]{\Lat\bigl(#1\bigr)}
\newcommand{\V}{\mathcal V}

\newcommand{\Ve}{\mathcal V^{\left<\varepsilon\right>}}
\newcommand{\svob}[2]{F_{\scriptstyle #1}\left[#2\right]}
\newcommand{\svobs}[2]{F^{(s)}_{\scriptstyle #1}\left[#2\right]}

\newcommand{\pnv}{\mathcal P_n\left(\mathcal V\right)}
\newcommand{\pnvv}{\mathcal P_n\left(\mathfrak N\right)}

\newcommand{\Ae}{\A^{\left<\varepsilon\right>}}
\newcommand{\ass}[3]{\left(#1,#2,#3\right)}
\newcommand{\A}{\mathcal A}
\newcommand{\B}{\mathcal B}
\title{\vspace{-24pt}
\bfseries Basic superranks for varieties of algebras}
\author{\bfseries Alexey Kuz'min\thanks{The first author is supported by the FAPESP 2010/51880--2 and the PNPD/CAPES/UFRN--PPgMAE},
Ivan Shestakov\thanks{The second author is supported by the FAPESP 2014/09310--5 and the CNPq 303916/2014--1}}
\date{}
\begin{document}
\maketitle

\begin{abstract}
We introduce the notion of basic superrank for varieties of algebras generalizing the notion of basic rank.
First we consider a number of varieties of nearly associative algebras over a field of characteristic~$0$
that have infinite basic ranks  and calculate  their basic superranks which turns out to be finite.
Namely we prove that the variety of alternative metabelian (solvable of index~$2$)
algebras possesses the two basic superranks $(1,1)$ and $(0,3)$;
the varieties  of Jordan and Malcev metabelian algebras have the unique basic superranks $(0,2)$ and $(1,1)$, respectively.
Furthermore, for arbitrary pair $(r,s)\neq (0,0)$ of nonnegative integers we provide a variety that has the unique basic superrank~$(r,s)$.
Finally, we construct some examples of nearly associative varieties that do not possess finite basic superranks.
\end{abstract}

\begin{list}{}{\rightmargin=\leftmargin\small}
\item\hspace{16pt}{\bfseries Key words:}
alternative algebra, Jordan algebra, Malcev algebra, metabelian algebra,
Grassmann algebra, superalgebra,
variety of algebras, basic rank of variety, basic superrank of variety, basic spectrum of variety.

\bigskip

\end{list}

\begin{flushright}
{\it Dedicated to Efim Isaakovich Zelmanov,\\ on the occasion of his  60th birthday}
\end{flushright}
\section*{Introduction}

Throughout the paper, all algebras are considered over a field of characteristic~$0$.
Let $\mathcal V$ be a variety of algebras and
$\mathcal V_r$ be a subvariety of $\mathcal V$ generated by the free $\mathcal V$-algebra of rank~$r$.
Then one can consider the chain
\[
\mathcal V_1\subseteq\mathcal V_2\subseteq\dots\subseteq\mathcal V_r\subseteq\dots\subseteq\mathcal V,
\]
where
$\mathcal V=\bigcup_{r}\mathcal V_r$.
If this chain stabilizes, then the minimal number~$r$ with the property
$\mathcal V_r=\mathcal V$
is called the \textit{basic rank} of the variety
$\mathcal V$
and is denoted by $\mathit{r_{b}}\left(\mathcal V\right)$ (see \cite{Malcev73}).
Otherwise, we say that $\mathcal V$ has the \textit{infinite basic rank}
$\mathit{r_{b}}\left(\mathcal V\right)=\aleph_0$.

Let us recall the main results on the basic ranks of the
\textit{varieties of associative}
($\mathrm{Assoc}$),
\textit{Lie}
($\mathrm{Lie}$),
\textit{alternative}
($\mathrm{Alt}$),
\textit{Malcev}
($\mathrm{Malc}$),
and some other \textit{algebras}.
It was first shown by A.~I.~Mal'cev~\cite{Malcev73} that
$r_b\left(\mathrm{Assoc}\right)=2$.
A.~I.~Shirshov~\cite{Zhevlakov-Slin'ko-Shestakov-Shirshov} proved that
$r_b\left(\mathrm{Lie}\right)=2$
and
$r_b\left(\mathrm{SJord}\right)=2$,
where $\mathrm{SJord}$ is the variety generated by all special Jordan algebras.
In 1958, A.~I.~Shirshov posed a problem on finding basic ranks
for alternative and some other varieties of nearly associative algebras \cite[Problem 1.159]{Dniester notebook}.
In 1977, the second author proved that $r_b\left(\mathrm{Alt}\right)=
r_b\left(\mathrm{Malc}\right)=\aleph_0$~\cite{Shestakov77,Zhevlakov-Slin'ko-Shestakov-Shirshov}.
The similar fact for the variety of algebras of type~$(-1,1)$ was established by S.~V.~Pchelintsev~\cite{Pchelintsev75}.
Note that the basic ranks of the varieties of Jordan and right alternative algebras are still unknown.

A proper subvariety of associative algebras can be of infinite basic rank as well.
For instance, so is the variety $\vr\mathrm G$ generated by the Grassmann algebra $\mathrm G$ on infinite number of generators,
or the variety defined by the identity $[x,y]^n=0,\ n>1$.

In 1986, A.~R.~Kemer~\cite{Kemer87,Kemer88} solved affirmatively the famous Specht problem~\cite{Specht50}
using the tool of superalgebras.
Recall that a variety $\mathcal V$ of algebras is called \textit{Spechtian}
or is said \textit{to have the Specht property}
if every algebra of $\mathcal V$ possesses a finite basis for its identities.
The Kemer Theorem states the Specht property of the variety of  associative algebras.
This result is obtained by certain reduction to the case of graded identities of finite dimensional superalgebras.
Namely it is proved that the ideal of identities of arbitrary associative algebra coincides with
the ideal of identities of the Grassmann envelope of some finite dimensional superalgebra.

This result suggests a generalization of the notion of basic rank.
Namely we shall say that a variety $\V$ has a finite basic superrank if
it can be generated by the Grassmann envelope of some finitely generated superalgebra.
Then  Kemer's result implies that every variety of associative algebras has a finite basic superrank.
This means that the notion of basic superrank is a more refined one then that of basic rank,
and we can distinguish varieties of infinite basic rank by their superranks.

The notion of basic superrank is the main subject of our paper.
First we consider a number of varieties of nearly associative algebras over a field of characteristic~$0$
that have infinite basic ranks  and calculate  their basic superranks which turns out to be finite.
Namely we prove that the variety of alternative metabelian (solvable of index~$2$)
algebras possesses the two basic superranks
$(1,1)$ and $(0,3)$;
the varieties  of Jordan and Malcev metabelian algebras have the unique basic superranks $(0,2)$ and $(1,1)$, respectively.
Furthermore, for arbitrary pair
$(r,s)\neq (0,0)$
of nonnegative integers we provide a variety that has the unique basic superrank~$(r,s)$.
Finally, we construct some examples of nearly associative varieties of algebras that do not possess finite basic superranks.

\section{Main definitions and results}%

Let
$\mathcal A=\mathcal A_0+\mathcal A_1$
be a \textit{superalgebra}
($\mathbb Z_2\text{-graded algebra}$) with the
\textit{even part $\mathcal A_0$}
and the
\textit{odd part $\mathcal A_1$}, i.~e.
$\mathcal A_i\mathcal A_j\subseteq\mathcal A_{i+j\;\,(\mathrm{mod} 2)}$
for
$i,j\in\{0,1\};$
${\mathrm G}$
be the \textit{Grassmann algebra} on a countable set of anticommuting generators
$\left\{e_1,e_2,\ldots\mid e_ie_j=-e_je_i\right\}$
with the natural $\mathbb Z_2\text{-grading}$
($\mathrm G_0$ and $\mathrm G_1$ are spanned by the words of even and, respectively, odd length on $\left\{e_i\right\}$).
The \textit{Grassmann envelope}
${\mathrm G}\left(\mathcal A\right)$
of a superalgebra $\mathcal A$ is the subalgebra\linebreak
$
{\mathrm G_0}\otimes \mathcal A_0+{\mathrm G_1}\otimes \mathcal A_1
$
of the tensor product
${\mathrm G}\otimes \mathcal A$.

It is well known that
${\mathrm G}\left(\mathcal A\right)$
satisfies a multilinear polynomial identity
$f=0$
if and only if
$\mathcal A$
satisfies the graded identity
$\Tilde{f}=0$
called the
\textit{superization of~$f=0$.}
Here,
$\Tilde{f}$ denotes the so-called
\textit{superpolynomial corresponding to $f$} and
we say that
$\mathcal A$ satisfies the \textit{superidentity}
$\Tilde{f}=0$.
The detailed descriptions of the process of constructing of superpolynomials
(the \textit{superizing process})
can be found in~\cite{Shestakov91,Shestakov93,Vaughan-Lee98,Zelmanov-Shestakov90}.
Roughly speaking, one should apply the so-called {\it Koszul rule} (or {\it Kaplansky rule}):
one should introduce the sign $(-1)^{ij}$ always when a variable of parity $i$ passes through a variable of parity $j$.

Let $\mathcal V$ be a variety of algebras defined by a system
$S$
of multilinear identities.
Recall that $\mathcal A$ is said to be a
$\mathcal V\textit{-superalgebra}$ if its Grassmann envelope lies in $\mathcal V$,
i.~e. if
$\mathcal A$
satisfies the system
$\Tilde{S}$
of all superidentities corresponding to the defining identities
of~$\mathcal V$.
Thus one can consider the set
$\Tilde{\mathcal V}$
of all
$\mathcal V\text{-superalgebras}$
as a
\textit{supervariety defined by the system
$\Tilde{S}$.}
It is clear that
$\Tilde{\mathcal V}$
can be generated by the free
$\mathcal V\text{-superalgebra}$ on a countable set of even and a countable set of odd generators.
Let
$\mathcal V_{r,s}$
be the
\textit{supervariety generated by the free
$\mathcal V\text{-superalgebra}$ on $r$ even and $s$ odd generators}
for
$(r,s)\neq (0,0)$ and
$\mathcal V_{0,0}=\bigl\{\left\{0\right\}\bigr\}$.
By
$\lt{\V}$
denote the lattice
\[
\begin{matrix}
                           &           & \mathcal V_{1,0}           & \subseteq  & \mathcal V_{2,0}            & \subseteq\cdots\subseteq  & \mathcal V_{r,0}                 & \subseteq\cdots \\
                           &           & \shortmid\hspace{-2pt}\cap &            & \shortmid\hspace{-2pt}\cap  &                           & \shortmid\hspace{-2pt}\cap       &                 \\
\mathcal V_{0,1}           & \subseteq & \mathcal V_{1,1}           & \subseteq  & \mathcal V_{2,1}            & \subseteq\cdots\subseteq  & \mathcal V_{r,1}                 & \subseteq\cdots \\
\shortmid\hspace{-2pt}\cap &           & \shortmid\hspace{-2pt}\cap &            & \shortmid\hspace{-2pt}\cap  &                           & \shortmid\hspace{-2pt}\cap       &                 \\
\vdots                     &           & \vdots                     &            & \vdots                      &                           & \vdots                           &                 \\
\shortmid\hspace{-2pt}\cap &           & \shortmid\hspace{-2pt}\cap &            & \shortmid\hspace{-2pt}\cap  &                           & \shortmid\hspace{-2pt}\cap       &                 \\
\mathcal V_{0,s}           & \subseteq & \mathcal V_{1,s}           & \subseteq  &\mathcal V_{2,s}             &\subseteq\cdots\subseteq   &\mathcal V_{r,s}                  & \subseteq\cdots \\
\shortmid\hspace{-2pt}\cap &           & \shortmid\hspace{-2pt}\cap &            & \shortmid\hspace{-2pt}\cap  &                           & \shortmid\hspace{-2pt}\cap       &                 \\
\hdotsfor{8}
\end{matrix}
\]
It is clear that
$\lt{\V}$
is partially well-ordered with the relation of inclusion.
A pair
$(r,s)$
is called the
\textit{basic superrank of the variety $\mathcal V$}
if
$\mathcal V_{r,s}$
is a minimal element of
$\lt{\V}$
with the property
$\mathcal V_{r,s}=\Tilde{\mathcal V}$.
If
$\mathcal V_{r,s}\neq\Tilde{\mathcal V}$
for all
$r,s$,
we say that the
\textit{basic superrank of $\mathcal V$ is infinite.}
The set of all possible finite basic superranks of
$\mathcal V$
is called the
\textit{basic spectrum of~$\mathcal V$} and is denoted by
$\bsp{\V}$.
For instance, one can easily check that
$\bsp{\mathrm{Assoc}}=\bsp{\mathrm{Lie}}=\{(2,0),(1,1),(0,2)\}$.
It is not hard to prove that if
$\V$
possesses at least one finite basic superrank
$(r,s)$, then
$\bsp{\V}$ is a finite set of at most $r+s+1$ elements.

\smallskip
The first examples of varieties of infinite basic superrank were constructed by M.~V.~Zaitsev~\cite{Zaitsev95,Zaitsev98}.

\smallskip

Let us fix some notations.
While writing down nonassociative monomials we use the symbol
$\cdot$
instead of parentheses to indicate the correct order of multiplication.
For instance, we write
$xy\cdot z$
instead of
$(xy)z$
and
$x\cdot yz$
instead of
$x(yz)$.
By
$$
\left[x,y\right]=xy-yx
\quad\text{and}\quad
x\circ y=xy+yx
$$
we denote, respectively,
the \textit{commutator} and the \textit{Jordan product} of the elements~$x,y$.
The notations
$$
\left(x,y,z\right)=xy\cdot z-x\cdot yz
\quad\text{and}\quad
\mathrm{J}\left(x,y,z\right)=xy\cdot z+yz\cdot x+zx\cdot y
$$
are used, respectively, for the \textit{associator} and  the \textit{Jacobian} of the elements~$x,y,z$.

Recall that the \textit{varieties}
$\mathrm{Alt}$, $\mathrm{Jord}$, and $\mathrm{Malc}$
\textit{of alternative, Jordan,} and \textit{Malcev algebras}
are defined by the following pairs of identities:
\begin{align}
\label{eq Alt}
\mathrm{Alt}:&\quad
\left(x,x,y\right)=0,\quad
\left(x,y,y\right)=0;\\
\label{eq Jord}
\mathrm{Jord}:&\quad
\left[x,y\right]=0,\quad
\left(x^2,y,x\right)=0;\\
\label{eq Malc}
\mathrm{Malc}:&\quad
x\circ y=0,\quad
{\mathrm J}\left(x,y,z\right)x={\mathrm J}\left(x,y,xz\right).
\end{align}
By
$\mathcal V^{(2)}$
we denote a
\textit{subvariety of all metabelian (solvable of index at most~$2$) algebras of a given variety~$\mathcal V$},
i.~e.
$\mathcal V^{(2)}$
is distinguished in
$\mathcal V$
by the identity
\begin{equation}\label{eq Metab}
xy\cdot zt=0.
\end{equation}
By
$\mathrm{NAlt}$
we denote a subvariety of all nilalgebras in
$\mathrm{Alt}$
of index at most~$3$,
i.~e.
$\mathrm{NAlt}$
is distinguished in
$\mathrm{Alt}$
by the identity
\begin{equation}\label{eq Nil3}
x^3=0.
\end{equation}

Let us formulate the results of the paper.
First, in Section~2, we describe the inclusions in the lattices
$\ltt{\mathrm{Alt}^{(2)}}$
and
$\ltt{\mathrm{NAlt}^{(2)}}$
and calculate the basic spectrums of these varieties.
Namely, we prove the following theorems.

\begin{theorem}\label{theorem NilAlt2}
$\bspp{\mathrm{NAlt}^{(2)}}=\bigl\{(0,1)\bigr\}$
and all inclusions
$\mathrm{NAlt}^{(2)}_n\subset\mathrm{NAlt}^{(2)}_{n+1}$
are strict.
\end{theorem}

\begin{theorem}\label{theorem Alt2}
For
$\mathcal V=\mathrm{Alt}^{(2)}$
we have
\[
\begin{matrix}
            &          & \mathcal V_{1,0}   & \subset  & \mathcal V_{2,0}   & \subset\cdots\subset & \mathcal V_{r,0}   & \subset  \cdots   \\
            &          &  \cap      &          & \cap       &                      & \cap       &        \\
 \mathcal V_{0,1}   & \subset  & \mathcal V_{1,1}   & =  & \mathcal V_{2,1}   & =\cdots= & \mathcal V_{r,1}   & =\cdots   \\
 \shortmid\shortmid &          &  \shortmid\shortmid      &          & \shortmid\shortmid       &                      & \shortmid\shortmid       &         \\
 \mathcal V_{0,2}   & \subset & \mathcal V_{1,2}   &=  & \mathcal V_{2,2}   & =\cdots= & \mathcal V_{r,2}   & =\cdots  \\
  \cap      &          &  \shortmid\shortmid      &          & \shortmid\shortmid       &                      & \shortmid\shortmid       &         \\
 \mathcal V_{0,3}   & = & \mathcal V_{1,3}   &=  & \mathcal V_{2,3}   & =\cdots= & \mathcal V_{r,3}   & =\cdots  \\
   \shortmid\shortmid       &          &  \shortmid\shortmid      &          & \shortmid\shortmid       &                      & \shortmid\shortmid       &         \\
 \mathcal V_{0,4}   & = & \mathcal V_{1,4}   &=  & \mathcal V_{2,4}   & =\cdots= & \mathcal V_{r,4}   & =\cdots  \\
  \shortmid\shortmid      &          &  \shortmid\shortmid      &          & \shortmid\shortmid       &                      & \shortmid\shortmid       &         \\
\hdotsfor{8}
\end{matrix}
\]
\end{theorem}

\begin{corollary}\label{corollary Alt 2}
$\bspp{\mathrm{Alt}^{(2)}}=\bigl\{(1,1),(0,3)\bigr\}$.
\end{corollary}

In Sections~3 and~4,
we describe the inclusions in the lattices
$\ltt{\mathrm{Jord}^{(2)}}$,
$\ltt{\mathrm{Malc}^{(2)}}$
and calculate the basic spectrums
$\bspp{\mathrm{Jord}^{(2)}}$,
$\bspp{\mathrm{Malc}^{(2)}}$,
respectively.

\begin{theorem}\label{theorem Jord2}
For
$\mathcal V=\mathrm{Jord}^{(2)}$
we have
\[
\begin{matrix}
            &          & \mathcal V_{1,0}   & \subset  & \mathcal V_{2,0}   & \subset\cdots\subset & \mathcal V_{r,0}   & \subset  \cdots   \\
            &          &  \cap      &          & \cap       &                      & \cap       &        \\
 \mathcal V_{0,1}   & \subset  & \mathcal V_{1,1}   &\subset   & \mathcal V_{2,1}   & \subset \cdots\subset  & \mathcal V_{r,1}   & \subset \cdots   \\
  \cap      &          &  \cap      &          & \cap       &                      & \cap       &         \\
 \mathcal V_{0,2}   & = & \mathcal V_{1,2}   &=  & \mathcal V_{2,2}   & =\cdots= & \mathcal V_{r,2}   & =\cdots  \\
   \shortmid\shortmid       &          &  \shortmid\shortmid      &          & \shortmid\shortmid       &                      & \shortmid\shortmid       &         \\
 \mathcal V_{0,3}   & = & \mathcal V_{1,3}   &=  & \mathcal V_{2,3}   & =\cdots= & \mathcal V_{r,3}   & =\cdots  \\
  \shortmid\shortmid      &          &  \shortmid\shortmid      &          & \shortmid\shortmid       &                      & \shortmid\shortmid       &         \\
\hdotsfor{8}
\end{matrix}
\]
\end{theorem}

\begin{corollary}\label{corollary Jord2}
$\bspp{\mathrm{Jord}^{(2)}}=\bigl\{(0,2)\bigr\}$.
\end{corollary}

\begin{theorem}\label{theorem Malc2}
For
$\mathcal V=\mathrm{Malc}^{(2)}$
we have
\[
\begin{matrix}
                 &           & \mathcal V_{1,0}   & \subset    & \mathcal V_{2,0}            & \subset\cdots\subset      & \mathcal V_{r,0}           & \subset  \cdots \\
                 &           & \cap               &            & \cap                        &                           & \cap                       &                 \\
\mathcal V_{0,1} & \subset   & \mathcal V_{1,1}   & =          & \mathcal V_{2,1}            & =\cdots=                  & \mathcal V_{r,1}           & =\cdots         \\
\cap             &           & \shortmid\shortmid &            & \shortmid\shortmid          &                           & \shortmid\shortmid         &                 \\
\mathcal V_{0,2} & \subset   & \mathcal V_{1,2}   & =          & \mathcal V_{2,2}            & =\cdots=                  & \mathcal V_{r,2}           & =\cdots         \\
\cap             &           & \shortmid\shortmid &            & \shortmid\shortmid          &                           & \shortmid\shortmid         &                 \\
\vdots           &           & \vdots             &            & \vdots                      &       \ddots              & \vdots                     &  \hfill         \\
\cap             &           & \shortmid\shortmid &            & \shortmid\shortmid          &                           & \shortmid\shortmid         &                 \\
\mathcal V_{0,s} & \subset   & \mathcal V_{1,s}   & =          &\mathcal V_{2,s}             & =\cdots=                  &\mathcal V_{r,s}            & =\cdots \\
\cap             &           & \shortmid\shortmid &            & \shortmid\shortmid          &                           & \shortmid\shortmid         &                 \\
\hdotsfor{8}
\end{matrix}
\]
\end{theorem}

\begin{corollary}\label{corollary Malc2}
$\bspp{\mathrm{Malc}^{(2)}}=\bigl\{(1,1)\bigr\}$.
\end{corollary}

\smallskip

Further, let
$\mathfrak M$
be the
\textit{variety of all metabelian algebras.}
In Section~5, we describe the inclusions in the lattice
$\lt{\mathfrak M}$
and provide an example of variety of unique arbitrary given finite basic superrank.

\begin{theorem}\label{theorem MetabInclusions}
The inclusion
$\mathfrak M_{r',s'}\subseteq\mathfrak M_{r,s}$
holds only if
$r'\leqslant r$
and
$s'\leqslant s$.
The equality
$\mathfrak M_{r',s'}=\mathfrak M_{r,s}$
takes place only if
$r'=r$
and
$s'=s$.
\end{theorem}

\begin{corollary}\label{corollary MetabInfin}
The  basic superrank of $\mathfrak M$ is infinite.
\end{corollary}

\begin{corollary}\label{corollary MetabArbit}
For an arbitrary pare
$(r,s)\neq(0,0)$
of nonnegative integers,
the variety of algebras generated by the Grassmann envelope of the free
$\mathfrak M_{r,s}\text{-superalgebra}$
has the unique basic superrank~$(r,s)$.
\end{corollary}

In Section~6, we provide some examples of varieties of  nearly associative algebras of infinite basic superrank.
For
$\varepsilon=\pm1$,
by
$\Ve$
denote a subvariety of
$\mathfrak M$
distinguished by the identities
\begin{align}
&\left(x,y,z\right)=\varepsilon\left(x,z,y\right),\label{eq EpsilonSymm}\\
&\bigl<\left<x,y\right>_{\varepsilon},z\bigr>_{\varepsilon}=0,\label{eq EpsilonNil2}
\end{align}
where
$\left<x,y\right>_{\varepsilon}=xy-\varepsilon yx$.
Thus,
$\mathcal V^{\left<+1\right>}$
is the \textit{variety of metabelian right symmetric algebras that are Lie-nilpotent of index at most $2$}
and
$\mathcal V^{\left<-1\right>}$
is the \textit{variety of metabelian right alternative algebras that are Jordan-nilpotent of index at most $2$.}

\begin{theorem}\label{theorem VeInfin}
The basic superrank of
$\Ve$
is infinite.
\end{theorem}

Finally, in Section~7, we suggest some open problems dealing with the introduced notions of basic superrank and basic spectrum
for varieties of algebras.

\subsection*{Common notations}%

Throughout the paper, we use the following notations. By
$\left[r\right]$
we denote the \textit{integer part of number~$r$};\,
$R_x$ and $L_x$ are \textit{operators of right} and
\textit{left multiplication}, respectively, \textit{by an
element~$x$};\, $T_x$ is a common notation for $R_x$ and $L_x$;\,
$T^*_x=\left\{
\begin{aligned}
L_x,\;\text{ if }\; T_x&=R_x,\\
R_x,\;\text{ if }\; T_x&=L_x;
\end{aligned}
\right.$\,
$X=\{x_1,x_2,\ldots\}$ is a \textit{countable set};\,
$X_n=\{x_1,x_2,\ldots,x_n\}$,
$n\in \mathbb N$;\,
$F$~is a
\textit{field of characteristic~$\chr F=0$};\,
$\svob{\mathcal V}{Y}$ is a \textit{free algebra of variety $\mathcal V$ on a set $Y$ of
free generators over $F$};\,
$\svobs{\mathcal V}{Z}$ is a \textit{free $\V\text{-superalgebra}$ on a set $Z$ of
free even and odd generators over $F$};\,
$\pnv$ is a \textit{subspace of
$\svob{\mathcal V}{X_n}$ of all multilinear polynomials of
degree~$n\geqslant 2$};\,
$(f)^{\mathrm T}$ is a \textit{$\mathrm{T}\text{-ideal}$ of
algebra $\svob{\mathcal V}{X}$ generated by the given polynomial
$f$};\,
${\mathrm S}_n$ is the \textit{symmetric group on the set $1,2,\dots,n$};\,
${\mathrm A}_n$ is the \textit{alternating subgroup of} ${\mathrm S}_n$;\,
${\mathrm C}_n$ is the \textit{subgroup of}
${\mathrm S}_n$
\textit{generated by the cycle~$\left(1\,2\ldots n\right)$};\,
$\left|\sigma\right|$ is a \textit{parity of permutation $\sigma\in{\mathrm S}_n$}, i.~e.
$\left|\sigma\right|= \left\{
\begin{aligned}
&0,\;\text{ if }\,\sigma\,\text{ is even,}\\
&1,\;\text{ if }\,\sigma\,\text{ is odd}.\\
\end{aligned}\right.$

In order to avoid complicated formulas while writing down the
elements of the space $\pnv$ we omit the indices of variables at
the operator symbols~$R,L$ and assume them to be arranged in the
ascending order.
For example, the notation
$\left(x_2x_4\right)LR^2$ means  the monomial
$\left(x_2x_4\right)L_{x_1}R_{x_3}R_{x_5}$.

\section{Alternative algebras}%
\label{Sec:AlternativeAlgebras}

Throughout this section, we set $\mathcal V=\mathrm{Alt}^{(2)}$
and $\mathfrak N=\mathrm{NAlt}^{(2)}$.

\subsection{Free $\mathcal V\text{-algebras}$}%
\label{SubSec:FreeAlt2algebras}

Recall that by the Artin Theorem~\cite[Chap.
2.3]{Zhevlakov-Slin'ko-Shestakov-Shirshov} every two-generated
alternative algebra is an associative one. It is also well known
that every alternative algebra satisfies the central Moufang
identity
\[
x\cdot yz\cdot x=xy\cdot zx.
\]
Thus by virtue of metability~\eqref{eq Metab} in the free algebra
$\svob{\mathcal V}{X}$, we have
\[
x\cdot yz\cdot x=0.
\]
Moreover, combining alternativity~\eqref{eq Alt} with~\eqref{eq
Metab}, we get
\[
x(x\cdot yz)=0,\quad (yz\cdot x)x=0,\quad (x,zt,y)=(zt\cdot
y)x=x(y\cdot zt).
\]
Therefore the relations
\[
T_xT^*_x=T_xT_x=0,\quad \left[L_x,R_y\right]=R_yR_x=L_yL_x
\]
hold for the operators of multiplication acting on
${\left(\svob{\mathcal V}{X}\right)}^2$. Using these relations,
one can prove the following

\begin{lemma}[\cite{Iltyakov82,Pchelintsev81}]\label{lemma-LinSpanAltPrelim}
The free algebra $\svob{\mathcal V}{X}$ is a linear span of the
monomials of the form
\[
\left(x_{i_1} x_{i_2}\right)T_{x_{i_3}}R_{x_{i_4}}\dots
R_{x_{i_n}},
\]
which are skew-symmetric with respect to their variables
$x_{i_3},x_{i_4},\dots,x_{i_n}$.
\end{lemma}
Note that in view of non-nilpotency of $\svob{\mathcal V}{X}$~(see
\cite{Dorofeev60,Shestakov91}), Lemma~\ref{lemma-LinSpanAltPrelim}
implies immediately the infiniteness of the basic rank of
$\mathcal V$.

Further, by  metability of $\svob{\mathcal V}{X}$, the Artin
Theorem yields the following

\begin{proposition}\label{proposition IdentSvobAlt}
The algebra $\svob{\mathcal V}{X}$ satisfies the identities
\begin{align}
x^3T_y&=0,\label{eq quasinil}\\
(xy)T_x R_y&=0.\label{eq quadrnilp}
\end{align}
\end{proposition}

Combining Lemmas~\ref{lemma-LinSpanAltPrelim} with
identities~\eqref{eq quasinil},~\eqref{eq quadrnilp}, it is not
hard to prove the following
\begin{lemma}\label{lemma SvobAltnNilp}
${\left(\svob{\mathcal V}{X_n}\right)}^{n+2}=0$.
\end{lemma}

\subsection{Free superalgebras on odd generators}%
\label{SubSec:FreeSuperalgebrasOnOdd}

Let us set
\[
\varphi(x_1,x_2,x_3)=\sum_{\sigma\in\mathrm{S}_3}\left(x_{\sigma(1)}x_{\sigma(2)}\right)x_{\sigma(3)}.
\]

\begin{lemma}\label{lemma MetabOneOdd}
The free metabelian superalgebra on one odd generator satisfies
the superidentity\, $\Tilde{\varphi}(x_1,x_2,x_3)=0$.
\end{lemma}

\begin{proof}
Let $\mathcal A$ be the free metabelian superalgebra on one odd
generator. Consider the value $\Tilde{\varphi}(x_1,x_2,x_3)$ for
arbitrary homogeneous elements $x_1,x_2,x_3\in\mathcal A$. By
metability of~$\mathcal A$, we may assume that at least two of the
elements $x_1,x_2,x_3$ are generators of~$\mathcal A$. But in this
case, taking into account that $\mathcal A$ has only one odd
generator, we obtain that the linear combination
$\Tilde{\varphi}(x_1,x_2,x_3)$ contains with every its monomial
$\alpha w$ $(\alpha=\pm1)$ the monomial $-\alpha w$. Hence,
$\Tilde{\varphi}(x_1,x_2,x_3)=0$ in $\mathcal A$.
\end{proof}

\begin{lemma}\label{lemma IntersecX3Pnv}
The intersection $ {\mathcal I} ={(x^3)}^{\mathrm T}\cap
\Bigl(\sum_{n=3}^{\infty}\pnv\Bigr) $ is spanned over $F$ by the
element\, $\varphi(x_1,x_2,x_3)$.
\end{lemma}

\begin{proof}
It is clear that $\varphi(x_1,x_2,x_3)\in \mathcal I$ as a
linearization of $x^3$. Furthermore,  by  \eqref{eq quasinil}, the
element $x^3$ and all its linearizations lie in the annihilator of
$\svob{\mathcal V}{X}$. On the other hand,
Lemma~\ref{lemma-LinSpanAltPrelim} and identity~\eqref{eq Metab}
yield $\varphi(w,x_2,x_3)=0$ for every element~$w\in\svob{\mathcal
V}{X}^2$. Therefore every element of $\mathcal I$ is proportional
to $\varphi(x_1,x_2,x_3)$.
\end{proof}

\begin{lemma}\label{lemma FreeVTwoOdd-Vv}
The free $\mathcal V\text{-superalgebra}$ on two odd generators is
an $\mathfrak N\text{-superalgebra}$.
\end{lemma}

\begin{proof}
Let $\mathcal A$ be the free $\mathcal V\text{-superalgebra}$ on
two odd generators. By  Lemma~\ref{lemma IntersecX3Pnv}, it
suffices to check $\Tilde{\varphi}(x_1,x_2,x_3)=0$ assuming that
$x_1,x_2,x_3$ are generators of $\mathcal A$. But in this case, at
least two of the elements $x_1,x_2,x_3$ coincide. Hence, with the
similar arguments as in Lemma~\ref{lemma MetabOneOdd}, one can
prove that $\Tilde{\varphi}(x_1,x_2,x_3)=0$.
\end{proof}

\subsection{Auxiliary $\mathfrak N\text{-superalgebra}$ on one odd generator}%
\label{SubSec:AuxiliaryVv-supOneOdd}

Let $U=F\cdot x$ be a superalgebra generated by an odd element~$x$
such that $x^2=0$. Consider a $\mathbb{Z}_2\text{-graded}$ space
$M=M_0+M_1$ over $F$ such that
\[
M_i=F[\varepsilon]\cdot a_i,\quad i=0,1,
\]
where $F[\varepsilon]$ is an algebraic extension of~$F$ with a
primitive 3-th root of~$1$, i.~e. such an element~$\varepsilon$
that $\varepsilon^2+\varepsilon+1=0$. It is clear that if
$\varepsilon\in F$, then $M$ is a $2\text{-dimensional}$ space
over $F$, otherwise, a $4\text{-dimensional}$ one. We define on
$M$ a structure of an $U\text{-superbimodule}$ such that the
action of the element~$x$ is given by the equalities
\[
a_i\cdot x = a_{1-i},\quad x\cdot a_i =
(i+\varepsilon)a_{1-i},\quad i=0,1.
\]
Consider the superalgebra $\mathcal A= U\dotplus M$ with the
$\mathbb{Z}_2\text{-grading}$
\[
\mathcal A=\mathcal A_0+\mathcal A_1,\quad \mathcal A_0=M_0,\quad
\mathcal A_1=U+M_1
\]
and the multiplication
\[
\left(u_1+m_1\right)\left(u_2+m_2\right)= u_1m_2+m_1u_2, \quad
u_1,u_2\in U,\quad m_1,m_2\in M.
\]
This superalgebra is called the \textit{split null extension of
$U$ by $M$}
(see~\cite{Trushina-Shestakov12,Zhevlakov-Slin'ko-Shestakov-Shirshov}).
It is known~\cite{Shestakov91,Trushina-Shestakov12} that $\mathcal
A$ is an alternative superalgebra.

\begin{lemma}\label{lemma A V'-SuperOneOdd}
$\mathcal A$ is an $\mathfrak N\text{-superalgebra}$ generated by
one odd element.
\end{lemma}

\begin{proof}
Let us show that $\mathcal A$ can be generated by the element
$y=a_1+x$. First we have
$$
y^2={\left(a_1+x\right)}^2=a_1\cdot x +x\cdot
a_1=(2+\varepsilon)a_0.
$$
Hence for $\varepsilon\in F$, we get
$$
a_0=\alpha y^2,\quad\; a_1=\alpha y^2\cdot
y,\quad\text{where}\quad \alpha=\frac{1-\varepsilon}{3}.
$$
Otherwise, we calculate
\begin{align*}
y^2\cdot y&=(2+\varepsilon)a_0\cdot x=(2+\varepsilon)a_1,\\
y\cdot
y^2&=x\cdot(2+\varepsilon)a_0=(2\varepsilon+\varepsilon^2)a_1=(\varepsilon-1)a_1.
\end{align*}
Considering the difference of the obtained equalities, we have
$$
3a_1=y^2\cdot y-y\cdot y^2.
$$
Therefore to express any element of $M$ with $y$ over $F$ it
suffices to use the relations
$$
a_0=a_1\cdot y,\qquad \varepsilon a_1=y\cdot a_0,\qquad
\varepsilon a_0=\varepsilon a_1\cdot y.
$$

To conclude the proof note that $\mathcal A$ is metabelian by
definition. Hence by Lemma~\ref{lemma MetabOneOdd}, it satisfies
identity~\eqref{eq Nil3}.
\end{proof}

\subsection{Additive basis of $\pnvv$}%
\label{SubSec:AdditiveBasisOfPnvv}

\begin{definitionsec}
{\upshape The \textit{basis words of $\pnvv$} are the polynomials
of the following types:}
\begin{align*}
1)&\enskip\left(x_1x_2\right)TR^{n-3},\\
2)&\enskip\left(x_1\circ x_i\right)TR^{n-3},\quad i=2,\dots,n.
\end{align*}
\end{definitionsec}

\begin{lemma}\label{lemma LinearSpanNilAlt}
The space $\pnvv$ is spanned by its basis words.
\end{lemma}

\begin{proof}
Let $u$ be an arbitrary monomial of $\pnvv$. By
Lemma~\ref{lemma-LinSpanAltPrelim} we may assume that $u$ has the
form
\[
u=\left(x_ix_j\right)TR^{n-3},\quad
i,j\in\left\{1,\dots,n\right\}.
\]
Consider the linear span $I_n$ of the polynomials of $\pnvv$ of
the form
\[
\left(x_i\circ x_j\right)TR^{n-3}.
\]
Using the linearization
\[
\sum_{\sigma\in\mathrm{C}_3}\left(x_{\sigma(1)}\circ
x_{\sigma(2)}\right)x_{\sigma(3)}=0
\]
of identity~\eqref{eq Nil3} it is not hard to show that $I_n$ is
the linear span of basis words of type~2). On the other hand, the
linearizations
\[
x\left(T_y\circ T_z\right)=\left(y\circ z\right) T^*_x
\]
of identities~\eqref{eq Alt} imply that the monomial $u$ is
skew-symmetric with respect to all its variables modulo~$I_n$.
\end{proof}

\begin{lemma}\label{lemma BasisWordsNilAlt}
Any nontrivial linear combination of  basis words of $\pnvv$ is
not an identity of ${\mathrm G}\left(\mathcal A\right)$.
\end{lemma}

\begin{proof}
Consider an arbitrary linear combination $f_n=g_n+h_n$ of basis
words of  $\pnvv$, where
\begin{align*}
g_n&= \alpha \left(x_1x_2\right)R^{n-2} +\alpha'
\left(x_1x_2\right)LR^{n-3},\\
h_n&= \sum_{i=2}^n\bigl( \beta_i \left(x_1\circ x_i\right)R^{n-2}
+\beta_i' \left(x_1\circ x_i\right)LR^{n-3}\bigr),
\end{align*}
for some scalars\, $\alpha,\alpha',\beta_j,\beta_j'\in F$. Suppose
that $\mathcal A$ satisfies the superidentity $\Tilde{f}_n=0$.
Then let us show that all the coefficients in $f_n$ are zero.

Fix $i>2$ and make in the identity $\Tilde{f}_n=0$ the
substitution $x_i=a_0,\ x_j=x$ for all $j\neq i$. Then we will get
\[
\beta_i(1+\varepsilon)a_0R^{n-3}-\beta_i'(1+\varepsilon)^2a_0R^{n-3}=0,
\]
which gives $\beta_{i}=(1+\varepsilon)\beta_{i}'$ for all $i>2$.
If we substitute $x_i=a_1$ instead of $a_0$, we will similarly get
$\beta_{i}=-\varepsilon\beta_{i}'$, which implies that
$\beta_{i}=\beta_{i}'=0$ for all $i>2$. Therefore, $h_n$ has a
form
\[
h_n=\beta (x_1\circ x_2)R^{n-2}+\beta'(x_1\circ x_2)LR^{n-3},\
\beta,\beta'\in F.
\]
Substituting now $x_i=a_0,\, i=1,2$ and $x_j=x$ for all $j\neq i$
we will get the two equalities
\begin{gather*}
\alpha a_0R^{n-3}-\alpha'(1+\varepsilon)a_0R^{n-3}+H_0=0,\\
\varepsilon(\alpha
a_0R^{n-3}-\alpha'(1+\varepsilon)a_0R^{n-3})+H_0=0,
\end{gather*}
where $H_0=\tilde h_n(a_0,x,\ldots,x)=\tilde h_n(x,a_0,\ldots,x)$.
This implies that $\alpha =(1+\varepsilon)\alpha'$ and $H_0=0$.
Similarly, making the substitutions $x_i=a_1,\, i=1,2$ and $x_j=x$
for all $j\neq i$ we get $\alpha=-\varepsilon \alpha'$ and
$H_1=0$, where $H_1=\tilde h_n(a_1,x,\ldots,x)=-\tilde
h_n(x,a_1,\ldots,x)$. Thus, $\alpha=\alpha'=0$. Finally, we have
\begin{gather*}
H_0=\beta(1+\varepsilon)a_0R^{n-3}-\beta'(1+\varepsilon)^2a_0R^{n-3}=0,\\
H_1=-\beta\varepsilon a_1R^{n-3}-\beta'\varepsilon^2a_1R^{n-3}=0,
\end{gather*}
which implies that $\beta=\beta'=0$. Therefore all the scalars in
$f_n$ are zero.
\end{proof}

\smallskip

By Lemmas~\ref{lemma A V'-SuperOneOdd}--\ref{lemma
BasisWordsNilAlt}, we obtain $\mathfrak N_{0,1}=\Tilde{\mathfrak
N}$.

\subsection{Index of nilpotency of $\svob{\mathfrak N}{X_n}$}%
\label{SubSec:NilpotencyIndexOfSvobNilpAltn}

\begin{lemma}\label{lemma IndexNilpSvobNilAltn}
The index of nilpotency of $\svob{\mathfrak N}{X_n}$ is equal
to~$n+2$.
\end{lemma}

\begin{proof}
By Lemma~\ref{lemma SvobAltnNilp}, the index of nilpotency of
$\svob{\mathfrak N}{X_n}$ is not more than~$n+2$. Let us provide a
nonzero element of $\svob{\mathfrak N}{X_n}$ of degree~$n+1$.
Consider the monomial
\[
w=x^2_1R_{x_2}\dots R_{x_n}.
\]
Note that the linearization of $w$ can be written as a basis word
of type~2) of ${\mathcal P}_{n+1}\left(\mathfrak N\right)$. Hence
by Lemmas~\ref{lemma A V'-SuperOneOdd} and~\ref{lemma
BasisWordsNilAlt}, $w\neq 0$ in $\svob{\mathfrak N}{X_n}$.
\end{proof}

Lemma~\ref{lemma IndexNilpSvobNilAltn} implies the strict
inclusions $\mathfrak N_n\subset\mathfrak N_{n+1}$
for~$n\in\mathbb N$.

Theorem~\ref{theorem NilAlt2} is proved.

\subsection{Auxiliary $\mathcal V\text{-superalgebras}$}%
\label{SubSec:AuxiliaryV-superalgebras}

Let $\mathcal A$ be the superalgebra defined in
Section~\ref{SubSec:AuxiliaryVv-supOneOdd}. Consider two
superalgebras $\mathcal B=\mathcal B_0+\mathcal B_1$ and $\mathcal
B'=\mathcal B'_0+\mathcal B'_1$ defined by the following
conditions:
\begin{enumerate}
\item
\[
\begin{aligned}
\mathcal B_0&=\mathcal A_0+ F\cdot e,            & \mathcal B_1&=\mathcal A_1+F\cdot ex+F\cdot xe+F\cdot exe,\\
\mathcal B'_0&=\mathcal A_0+F\cdot yx+F\cdot xz,& \mathcal
B'_1&=\mathcal A_1+F\cdot y+F\cdot z+F\cdot yxz;
\end{aligned}
\]
\item $\mathcal A$ is a subalgebra of $\mathcal B$ and $\mathcal
B'$; \item all nonzero products of basis elements of $\mathcal B$
and $\mathcal B'$, in the case when at least one of the factors
doesn't lie in $\mathcal A$, are the following:
\begin{gather*}
e\cdot x = ex,\quad x\cdot e= xe,\quad ex\cdot e= e\cdot xe= exe;\\
y\cdot x = yx,\quad x\cdot z= xz,\quad yx\cdot z= y\cdot xz= yxz.
\end{gather*}
\end{enumerate}

\begin{lemma}\label{lemma B V-Super}
$\mathcal B$ is a $\mathcal V\text{-superalgebra}$ generated by
one even and one odd elements; $\mathcal B'$ is a $\mathcal
V\text{-superalgebra}$ generated by three odd elements.
\end{lemma}

\begin{proof}
It is clear that $\mathcal B$ can be generated by the elements
$e,x$ and $\mathcal B'$ can be generated by the elements $x,y,z$.
Thus by virtue of Lemma~\ref{lemma A V'-SuperOneOdd}, it remains
to notice that if some product $\rho\neq0$ of three basis elements
of $\mathcal B$ or $\mathcal B'$ contains a factor not lying
in~$\A$, then $\rho$ is associative and lies in the annihilator of
the corresponding algebra $\mathcal B$ or $\mathcal B'$.
\end{proof}

\subsection{Additive basis of $\pnv$}%
\label{SubSec:AdditiveBasisOfPnv}

\begin{definitionsec}
{\upshape The \textit{basis words of $\pnv$} are all basis words
of $\pnvv$ and the polynomial~$\varphi(x_1,x_2,x_3)$.}
\end{definitionsec}

It follows easily from Lemmas~\ref{lemma IntersecX3Pnv}
and~\ref{lemma LinearSpanNilAlt} that the space $\pnv$ is  spanned
by its basis words.

\begin{lemma}\label{lemma BasisWordsAlt}
Any nontrivial linear combination of  basis words of $\pnv$ is not
an identity of either ${\mathrm G}\left(\mathcal B\right)$ or
${\mathrm G}\left(\mathcal B'\right)$.
\end{lemma}

\begin{proof}
By definition of ${\mathrm G}\left(\mathcal B\right)$ and
${\mathrm G}\left(\mathcal B'\right)$, taking into account
Lemma~\ref{lemma BasisWordsNilAlt}, it suffices to check that
$\Tilde{\varphi}(x_1,x_2,x_3)$ takes a nonzero value on some
elements of $\mathcal B$ and $\mathcal B'$. Indeed,
\[
\Tilde{\varphi}(e,x,e)=2\left(e\cdot x\right)\cdot e=2\,exe\neq
0;\qquad \Tilde{\varphi}(y,x,z)=\left(y\cdot x\right)\cdot
z=yxz\neq 0.
\]
\end{proof}

\smallskip

By Lemmas~\ref{lemma B V-Super} and \ref{lemma BasisWordsAlt}, we
have $ \mathcal V_{1,1}=\mathcal V_{0,3}=\Tilde{\mathcal V}. $
Combining Lemmas~\ref{lemma FreeVTwoOdd-Vv} and~\ref{lemma
BasisWordsAlt} with Theorem~\ref{theorem NilAlt2}, we obtain $
\mathcal V_{0,2}\subseteq\Tilde{\mathfrak N}=\mathcal
V_{0,1}\subset\Tilde{\mathcal V}. $

\smallskip

\subsection{Strictness of  inclusions $\V_n\subset\V_{n+1}$}

Finally, to complete the proof of Theorem~\ref{theorem Alt2}, it
suffices to verify the strictness of all inclusions in the first
row of the lattice $\lt{\V}$. In fact, it follows from
Lemmas~\ref{lemma IntersecX3Pnv} and~\ref{lemma
IndexNilpSvobNilAltn} that the index of nilpotency of
$\svob{\mathcal V}{X_n}$ for~$n\geqslant2$ is equal to~$n+2$.
Therefore, $\mathcal V_n\neq \mathcal V_{n+1}$ for all
$n\geqslant2$. It remains to notice that the variety $\V_1$ is
commutative and $\V_2$ is not.

Theorem~\ref{theorem Alt2} is proved.

\section{Jordan algebras}%
\label{Sec:JordanAlgebras}

Throughout this section, we set $\mathcal V=\mathrm{Jord}^{(2)}$.

\subsection{Additive basis of the free $\mathcal V\text{-algebra}$}%
\label{SubSec:AdditiveBasisOfTheFreeJordAlgebra}

It is
known~\cite{Drensky-Rashkova89,Jacobson68,Pchelintsev81,Sverchkov83}
that an additive basis of the free algebra $\svob{\mathcal V}{X}$
can be formed by the following monomials
\[
\left(x_{k}x_{i_1}\right)R_{x_{j_1}}R_{x_{i_2}}R_{x_{j_2}}\dots
R_{x_{i_t}}R'_{x_{j_t}},
\]
where
\[
k\geqslant i_1<i_2<\dots<i_t,\quad j_1<j_2<\dots<j_t,
\]
and the symbol $'$ means that the operator $R_{x_{j_t}}$ is absent
when the degree of monomial is even. In what follows, we use this
basis with no comments.

\subsection{Nilpotency of the free algebra $\svob{\mathcal V}{X_n}$}%
\label{SubSec:NilpotencyOfTheFreeAlgebra}

\begin{proposition}\label{proposition IdentSvobJord}
The algebra $\svob{\mathcal V}{X}$ satisfies the identities
\begin{align}
\label{eq JordR}
x^2R_yR_x&=0,\\
\label{eq JordCos}
(zt)R_xR_yR_x&=0,\\
\label{eq JordKvNilp} (zx)R_yR_xR_tR_z&=0.
\end{align}
\end{proposition}

\begin{proof}
First we stress that~\eqref{eq JordR} is a direct consequence of
metability and~\eqref{eq Jord}. Further, taking into account
commutativity, we write the partial linearization of~\eqref{eq
JordR} in the form
\begin{equation}\label{eq LinJordR}
2(zx)R_yR_x+x^2R_yR_z=0.
\end{equation}
Then by setting $z:=zt$ in~\eqref{eq LinJordR}, in view of
metability, we get~\eqref{eq JordCos}. Finally,
multiplying~\eqref{eq LinJordR} by $R_tR_z$ and applying~\eqref{eq
JordCos}, we obtain~\eqref{eq JordKvNilp}.
\end{proof}

\begin{lemma}\label{lemma IndexNilpSvobJordn}
The algebra $\svob{\V}{X_n}$ is nilpotent of index~$2n+2$.
\end{lemma}

\begin{proof}
Identities~\eqref{eq JordR}--\eqref{eq JordKvNilp} imply
immediately that ${\left(\svob{\mathcal V}{X_n}\right)}^{2n+2}=0$.
Therefore it remains to note that the element
\[
x^2_1R_{x_1}R^2_{x_2}R^2_{x_3}\dots R^2_{x_n}
\]
of the additive basis of $\svob{\V}{X}$ is a nonzero element of
${\left(\svob{\mathcal V}{X_n}\right)}^{2n+1}$.
\end{proof}

Lemma~\ref{lemma IndexNilpSvobJordn} yields that all inclusions
$\V_n\subset\V_{n+1}$ are strict and, consequently, the basic rank
of $\V$ is infinite.

\subsection{Estimate of basic superrank of $\V$}%
\label{SubSec:EqualityV0,2V}

First note that superizing~\eqref{eq JordCos}, one can prove the
following
\begin{proposition}\label{proposition RelJordSvobSuper}
Let $\svobs{\V}{Z}$ be a free $\V\text{-superalgebra}$ on an
arbitrary set $Z$ of even and odd generators. Then the operators
of multiplication acting on ${\bigl(\svobs{\mathcal
V}{Z}\bigr)}^2$ satisfy the relation
\begin{equation}\label{rel SupJordCos}
R_xR_yR_z={(-1)}^{|x||y|+|x||z|+|y||z|+1}R_zR_yR_x,
\end{equation}
where $x,y,z\in Z$ and $|x|$ denotes the parity of~$x$.
\end{proposition}

\smallskip

Let $U=U_0+U_1$ be the superalgebra
\[
U_0=\{0\},\quad\; U_1=F\cdot x+F\cdot y
\]
with null multiplication and $M=M_0+M_1$ be the vector space
\[
M_0=F\cdot a,\quad\; M_1=F\cdot v.
\]
Consider a split null extension $\A= U\dotplus M$ with a
supercommutative multiplication such that all nonzero products of
the basis elements of $\A$, up to the order of factors, are the
following:
\[
a\cdot x=v,\quad\; v\cdot y=a.
\]
The supercommutativity rule means that even elements commute with all elements of superalgebra
but products of two odd elements are anticommutative.

\begin{lemma}\label{lemma-A-Jord2super}
$\A$ is a $\V\text{-superalgebra}$ generated by two odd elements.
\end{lemma}

\begin{proof}
By definition, $\A$ is metabelian, supercommutative, and can be
generated by the elements $v+x$ and $y$. It remains to check that
$\A$ is a Jordan superalgebra. Note that by construction of $\A$
it suffices to verify that relation~\eqref{rel SupJordCos} holds
in $\A$. But it follows trivially from the definition of
multiplication in $\A$.\footnote{One can also note that $\A$ is a
special Jordan superalgebra isomorphic to a subalgebra of the
matrix superalgebra $M^{(+)}_{2,2}$ (see~\cite{Shestakov91}) with
the generators $e_{13}-e_{24}+e_{31}-e_{42}$ and
$e_{13}+e_{14}+e_{24}+e_{31}+e_{32}+e_{42}$.}
\end{proof}

\begin{lemma}\label{lemma-UpperBoundJord}
The variety $\V$ is generated by ${\mathrm G}\left(\mathcal
A\right)$.
\end{lemma}

\begin{proof}
In view of Lemma~\ref{lemma-A-Jord2super} it suffices to prove
that ${\mathrm G}\left(\mathcal A\right)$ doesn't satisfy any
nontrivial identity in $\V$. Consider an arbitrary linear
combination of basis monomials of $\pnv$:
\[
f_n= \sum_{I}\alpha_{I}
\left(x_{k}x_{i_1}\right)R_{x_{j_1}}R_{x_{i_2}}R_{x_{j_2}}\dots
R_{x_{i_t}}R'_{x_{j_t}},
\]
where $t=\left[\frac{n}{2}\right]$ and $I$ runs all possible sets
$k,i_1,\dots,i_t,j_1,\dots,j_t$ of indices such that
\[
k>i_1<i_2<\dots<i_t,\quad j_1<j_2<\dots<j_t.
\]
Suppose that $\A$ satisfies the superidentity $\Tilde{f}_n=0$. Let
us show that all the scalars in $f_n$ are zero. We fix
$I=\left\{k,i_1,\dots,i_t,j_1,\dots,j_t\right\}$ and make the
substitution
\[
x_k=a,\quad x_{i_1}=\dots=x_{i_t}=x,\quad x_{j_1}=\dots=x_{j_t}=y.
\]
Then it is not hard to see that $\Tilde{f}_n$ turns out to be
proportional with the coefficient $\pm\alpha_{I}$ to the element
\[
(a\cdot x)R_yR_x\cdots= \left\{
\begin{aligned}
v,\quad&\text{if $n$ is even},\\
a,\quad&\text{if $n$ is odd}
\end{aligned}\right.
\]
that is nonzero in $\A$. Therefore, $\alpha_{I}=0$.
\end{proof}

Lemma~\ref{lemma-UpperBoundJord} implies that
$\V_{0,2}=\Tilde{\V}$. Thus, to complete the proof of
Theorem~\ref{theorem Jord2}, it remains to show that the chain
\[
\V_{0,1}\subset\V_{1,1}\subset\V_{2,1}\subset\dots\subset\V_{r,1}\subset\dots\subset\Tilde{\V}
\]
ascends strictly and all the inclusions $\V_{n,0}\subset\V_{n,1}$
are also strict.

\subsection{Strictness of the inclusions $\V_{n,0}\subset\V_{n,1}$ and $\V_{n-1,1}\subset\V_{n,1}$}%
\label{SubSec:InclusionsJordColumnsLines}

Let $U^{(n)}=U^{(n)}_0+U^{(n)}_1$ be the superalgebra
\[
U^{(n)}_0=\sum_{i=1}^n F\cdot e_i,\quad\; U^{(n)}_1=F\cdot y
\]
with null multiplication and $\A^{(n)}=\A^{(n)}_0+\A^{(n)}_1$ be
an associative superalgebra with the unit~$\mathbf 1$ generated by
the even elements $\mathbf{1},\mathbf{e}_1,\dots,\mathbf{e}_n$ and
one odd element $\mathbf y$ with the defining relations
\[
\mathbf{e}_i\mathbf{e}_j=0,\quad\; \mathbf{y}^2=\mathbf1, \quad\;
\mathbf{e}_i\mathbf{y}\mathbf{e}_j=-\mathbf{e}_j\mathbf{y}\mathbf{e}_i.
\]
Consider a split null extension $\B^{(n)}= U^{(n)}\dotplus
\A^{(n)}$ with a supercommutative multiplication induced by the
actions
\begin{gather*}
\mathbf{1}\cdot e_i=0, \quad\; \mathbf{a}\cdot
e_i=\mathbf{a}\mathbf{e}_i, \quad
\mathbf1\neq\mathbf{a}\in\A^{(n)},\\
\mathbf{b}\cdot y=\mathbf{b}\mathbf{y}, \quad
\mathbf{b}\in\A^{(n)}.
\end{gather*}

\begin{lemma}\label{lemma-An-Jord2super}
$\B^{(n)}$ is a $\V\text{-superalgebra}$ generated by $n$ even and
one odd elements.
\end{lemma}

\begin{proof}
By definition, $\B^{(n)}$ is metabelian and supercommutative.
Moreover, it is not hard to see that $\B^{(n)}$ can be generated
by the elements $\mathbf1+e_1,e_2,\dots,e_n,y$. Thus it remains to
prove that $\B^{(n)}$ is a Jordan superalgebra. Actually by
definition of multiplication in $\B^{(n)}$, taking into account
that $\left(\B^{(n)}\right)^2\subseteq\A^{(n)}$, it suffices to
verify that relation~\eqref{rel SupJordCos} holds for the
operators of the form $R_{u_1}R_{u_2}R_{u_3}$ acting on $\A^{(n)}$
for $u_i\in\{e_1,\dots,e_n,y\}$. Indeed, besides the trivial case
$u_1=u_3=y$, we check that $R_{e_i}R_{e_j}$ annihilates
$\A^{(n)}$:
\[
\mathbf{a}R_{e_i}R_{e_j}=\mathbf{a}\mathbf{e}_i\mathbf{e}_j=0;
\]
$R^2_y$ acts on $\A^{(n)}$ identically:
\[
\mathbf{a}R^2_y=\mathbf{a}\mathbf{y}^2=\mathbf{a};
\]
and the action of $R_{e_i}R_yR_{e_j}$ on $\A^{(n)}$ is
skew-symmetric with respect to $e_i,e_j$:
\[
\mathbf{a}R_{e_i}R_yR_{e_j}
=\mathbf{a}\mathbf{e}_i\mathbf{y}\mathbf{e}_j
=-\mathbf{a}\mathbf{e}_j\mathbf{y}\mathbf{e}_i
=-\mathbf{a}R_{e_j}R_yR_{e_i}.
\]
\end{proof}

\smallskip

By Lemmas~\ref{lemma IndexNilpSvobJordn}
and~\ref{lemma-An-Jord2super}, in view of non-nilpotency of
$\B^{(n)}$, we obtain $\V_{n,0}\subset\V_{n,1}$.

\medskip

Further, let us denote by
$f_n=f_n\left(a,b,x_1,\dots,x_{2n}\right)$ the polynomial
\[
f_n=\left(ab\right)\left(R_{x_1}\circ R_{x_2}\right)\dots
\left(R_{x_{2n-1}}\circ R_{x_{2n}}\right),
\]
where $R_x\circ R_y=R_xR_y+R_yR_x$.

\begin{lemma}\label{lemma-LowBoundJord}
The free $\V\text{-superalgebra}$ on $n-1$ even and one odd
generators satisfies the superidentity $\Tilde{f}_{n}=0$.
\end{lemma}

\begin{proof}
Let $A_{n-1,1}$ be the free $\V\text{-superalgebra}$ on $n-1$ even
and one odd generators. Consider a value of $\Tilde{f}_{n}$ on
some homogeneous elements
$\tilde{a},\tilde{b},\tilde{x}_1,\dots,\tilde{x}_{2n}\in
A_{n-1,1}$. In view of metability, we may assume that all the
elements of the set $S=\{\tilde{x}_1,\dots,\tilde{x}_{2n}\}$ are
generators of~$A_{n-1,1}$. Thus by definition of $f_n$, it is
clear that $\Tilde{f}_{1}=0$ in $A_{0,1}$.

Let us prove by induction on $n$ that $\Tilde{f}_{n}=0$ in
$A_{n-1,1}$. For $n\geqslant2$, by $e_1,\dots,e_{n-1}$ we denote
the even generators of $A_{n-1,1}$ and $y$ denotes its odd
generator. For a monomial
$w=(\tilde{a}\tilde{b})R_{\tilde{x}_1}R_{\tilde{x}_2}\dots
R_{\tilde{x}_{2n}}$ we set
$S(w)=\left\{\tilde{x}_1,\tilde{x}_3,\dots,\tilde{x}_{2n-1}\right\}$
and $\Bar{S}(w)=S\setminus S(w)$. Assume that $w\neq0$ in
$A_{n-1,1}$. Then it follows from~\eqref{rel SupJordCos} that
every~$e_i$ can be presented only once in each set $S(w)$ and
$\Bar{S}(w)$. Thus, $\Tilde{f}_{n}$ can be nonzero only if
every~$e_i$ is included in $S$ not more than twice. On the other
hand, if some~$e_i$ is included in $S$ only once, then one can
represent $\Tilde{f}_{n}$ in the form
\[\
\Tilde{f}_{n}=\pm \Tilde{f}_{n-1}
\bigl(\tilde{a},\tilde{b},\tilde{x}_1,\dots,\tilde{x}_{2j-2},\tilde{x}_{2j+1},\dots,\tilde{x}_{2n-2}\bigr)
\left(R_{\tilde{x}_{2j-1}}\circ R_{\tilde{x}_{2j}}\right),
\]
where $e_i\in\left\{\tilde{x}_{2j-1},\tilde{x}_{2j}\right\}$. In
this case, by inductive hypothesis, we have $\Tilde{f}_{n}=0$.
Therefore, it suffices to consider the case when every~$e_i$ is
included in $S$ twice exactly and $w\neq0$. This assumption yields
that the sets $S(w)$ and $\Bar{S}(w)$ consist of the same elements
$e_1,\dots,e_{n-1},y$. Consequently, except of~$w$, there is only
one more nonzero monomial $w'$ in the linear combination
$\Tilde{f}_{n}$ of the form
\[
w'=(\tilde{a}\tilde{b})R_{\tilde{x}_2}R_{\tilde{x}_1}R_{\tilde{x}_4}R_{\tilde{x}_3}\dots
R_{\tilde{x}_{2n}}R_{\tilde{x}_{2n-1}}.
\]
Hence we have
\[
\Tilde{f}_{n}=w+{(-1)}^{|\tilde{x}_1||\tilde{x}_2|+|\tilde{x}_3||\tilde{x}_4|+\dots+|\tilde{x}_{2n-1}||\tilde{x}_{2n}|}w'=w+w'.
\]
By virtue of~\eqref{rel SupJordCos}, taking into account the
equalities $S(w')=\Bar{S}(w)$ and $\Bar{S}(w')=S(w)$, it is not
hard to see that $w'$ is proportional to $w$. Thus it remains to
prove that the coefficient of this proportionality is equal to
$-1$. Let $\sigma$ be the permutation that transforms $\Bar{S}(w)$
into $S(w)$. It is clear that one can transform $w'$ into $w$
acting by $\sigma$ on $S(w')$ and by $\sigma^{-1}$ on
$\Bar{S}(w')$. While that, a scalar $\pm 1$ appearing after this
transformation will not depend on the parity of $\sigma$, but will
depend only on number of transpositions made by the odd elements.
Taking into account that $\sigma(y)\neq y$, it is not hard to
understand that such a transposition will be only one. Therefore,
$w'=-w$ and, consequently, $\Tilde{f}_{n}=0$.
\end{proof}

By Lemmas~\ref{lemma-An-Jord2super} and~\ref{lemma-LowBoundJord},
to prove the strictness of inclusions $\V_{n-1,1}\subset\V_{n,1}$
it suffices to verify that $\Tilde{f}_{n}$ takes a nonzero value
in $\B^{(n)}$. Indeed,
\begin{multline*}
\Tilde{f}_{n}\left(\mathbf{1},y,e_1,y,e_2,y,\dots,e_{n},y\right)
=\mathbf{y}\left(R_{e_1}\circ R_{y}\right)\left(R_{e_2}\circ R_{y}\right)\dots\left(R_{e_{n}}\circ R_{y}\right)=\\
=\left(\mathbf{y}\mathbf{e}_1\cdot y+\mathbf{1}\cdot e_1\right)\left(R_{e_2}\circ R_{y}\right)\dots\left(R_{e_{n}}\circ R_{y}\right)=\\
=\mathbf{y}\mathbf{e}_1\mathbf{y}\left(R_{e_2}\circ
R_{y}\right)\dots\left(R_{e_{n}}\circ R_{y}\right)=\dots=
\mathbf{y}\mathbf{e}_1\mathbf{y}\mathbf{e}_2\mathbf{y}\dots\mathbf{e}_n\mathbf{y}\neq0.
\end{multline*}

\smallskip

Theorem~\ref{theorem Jord2} is proved.

\begin{remark}
{\upshape Note that Theorem~\ref{theorem Jord2} gives a more
detailed description of the inclusions in the lattice $\lt{\V}$
than one needs to deduce the uniqueness of the basic superrank
$(0,2)$ for $\V$. Actually proving that
$\bsp{\V}=\bigl\{(0,2)\bigr\}$ we could restrict with establishing
the equality $\V_{0,2}=\Tilde{\V}$ and the strict inclusions
$\V_{r,1}\subset\Tilde{\V}$ for $r=0,1,2,\dots$ Namely, in view of
Lemmas~\ref{lemma-A-Jord2super} and~\ref{lemma-LowBoundJord}, it
is enough to check that the superpolynomial $\Tilde{f_n}$ doesn't
vanish on some elements of the superalgebra $\A$.}
\end{remark}

\section{Malcev algebras}%
\label{Sec:MalcevAlgebras}

Throughout this section, we set $\mathcal V=\mathrm{Malc}^{(2)}$.

\subsection{Preliminary identities}%
\label{SubSec:PreliminaryIdentities}

It's well-known~\cite{Sagle61} that an anticommutative algebra
over a field of characteristic distinct from~$2$ is a Malcev one
if and only if it satisfies the Sagle identity
$$
\sum_{\sigma\in\mathrm{C}_4}\left(x_{\sigma(1)}x_{\sigma(2)}\right)R_{x_{\sigma(3)}}R_{x_{\sigma(4)}}
=\left(x_1x_3\right)\left(x_2x_4\right).
$$
By virtue of metability, the Sagle identity gets the form
\begin{equation}\label{eq Sagle}
\sum_{\sigma\in\mathrm{C}_4}\left(x_{\sigma(1)}x_{\sigma(2)}\right)R_{x_{\sigma(3)}}R_{x_{\sigma(4)}}=0.
\end{equation}
For $x_4=w\in\bigl(\svob{\V}{X}\bigr)^2$, identity~\eqref{eq
Sagle} implies
\begin{equation}\label{eq wSagle}
wR_{x_1}R_{x_2}R_{x_3}=wR_{x_3}R_{x_1}R_{x_2}.
\end{equation}
Moreover, combining~\eqref{eq Sagle} with anticommutativity and
taking into account that\linebreak $\chr F\neq 2$, we obtain
\begin{equation}\label{eq SagleSqr}
(xy)\left[R_x,R_y\right]=0.
\end{equation}
Finally, applying~\eqref{eq wSagle} and~\eqref{eq SagleSqr}, we
have
\begin{equation}\label{eq SagleGen}
(xy)\rho R_x\eta R_y=(xy)\rho R_y\eta R_x,
\end{equation}
for any operator words $\rho,\eta$.

\subsection{Auxiliary  $\V\text{-superalgebra}$}%
\label{SubSec:AuxiliaryMalcsuperalgebra}

Let $U=U_0+U_1$ be the superalgebra
\[
U_0=F\cdot e,\quad\; U_1=F\cdot y
\]
with null multiplication and $M=M_0+M_1$ be the vector space
\[
M_0=F\cdot a,\quad\; M_1=F\cdot v+ F\cdot w.
\]
Consider a split null extension $\A=U\dotplus M$ with a
superanticommutative multiplication such that all nonzero products
of the basis elements of $\A$, up to the order of factors, are the
following:
\[
a\cdot y=v,\quad v\cdot y=a,\quad w\cdot e=w.
\]
The superanticommutativity rule means that
even elements anticommute  with all elements of superalgebra
and odd elements commute to each other.

\begin{lemma}\label{lemma-A-Malc2super}
$\A$ is a $\V\text{-superalgebra}$ on one even and one odd
generators.
\end{lemma}

\begin{proof}
By definition, $\A$ is metabelian and superanticommutative.
Moreover, it is not hard to see that $\A$ can be generated by the
even element~$a+e$ and by the odd element $w+y$. It remains to
check that $\A$ is a Malcev superalgebra. By construction of $\A$,
it suffices to verify that the superization of~\eqref{eq wSagle}
holds in $\A$, i.~e. that the action of an operator
$R_{z_1}R_{z_2}R_{z_3}$ on $M$, for homogeneous $z_i\in U$,
satisfies the relation
\[
R_{z_1}R_{z_2}R_{z_3}={(-1)}^{|z_1||z_2|+|z_1||z_3|}R_{z_2}R_{z_3}R_{z_1}.
\]
Actually it is not hard to see that this relation holds trivially
in both possible nonzero cases $z_1=z_2=z_3=y$ and
$z_1=z_2=z_3=e$.
\end{proof}

\subsection{Additive basis of $\pnv$}%
\label{SubSec:BasicWordsOfMalc} By virtue of anticommutativity, we
consider the space $\pnv$ only for $n\geqslant3$. Following the
fixed above notations, we write down the monomials of $\pnv$
omitting some uniquely restored indices that are assumed to be
arranged in the ascending order.

\begin{definitionsec}
{\upshape The \textit{basis words of $\pnv$} are the polynomials
of the following types:
\begin{align*}
1)&\enskip
\left(x_1 x_2\right)R^{n-2},\quad \left(x_2 x_3\right)R^{n-2},\quad \left(x_3 x_1\right)R^{n-2},\\
2)&\enskip
\left(x_1 x_i\right)R^{n-2},\quad i=4,\dots,n,\\
3)&\enskip \left(x_1 x_2\right)\left(R_{x_3}\circ
R_{x_4}\right)R^{n-4}, \quad \left(x_2
x_3\right)\left(R_{x_1}\circ R_{x_4}\right)R^{n-4},
\quad \left(x_3 x_1\right)\left(R_{x_2}\circ R_{x_4}\right)R^{n-4},\\
4)&\enskip \left(x_1 x_i\right)\left(R_{x_2}\circ
R_{x_3}\right)R^{n-4},\quad i=4,\dots,n,
\end{align*}
where $R_x\circ R_y=R_xR_y+R_yR_x$.}
\end{definitionsec}

\begin{lemma}\label{lemma-LinearSpanMalc}
$\pnv$ is spanned by its basis words.
\end{lemma}

\begin{proof}
In view of anticommutativity, it is clear that $\mathcal
P_3\left(\V\right)$ is spanned by its basis words of type~$1)$.

Let $I_n$ be the linear span of basis words of $\pnv$ for
$n\geqslant4$. Then using~\eqref{eq Sagle} and anticommutativity,
we have
\begin{multline*}
\left(x_2 x_4\right)R_{x_1}R_{x_3}= -\left(x_4
x_1\right)R_{x_3}R_{x_2} -\left(x_1 x_3\right)R_{x_2}R_{x_4}
-\left(x_3 x_2\right)R_{x_4}R_{x_1}=\\
\left(x_1 x_4\right)\left(R_{x_2}\circ R_{x_3}\right) -\left(x_1
x_4\right)R_{x_2}R_{x_3}
-\left(x_1 x_3\right)R_{x_2}R_{x_4}+\\
+\left(x_2 x_3\right)\left(R_{x_1}\circ R_{x_4}\right) -\left(x_2
x_3\right)R_{x_1}R_{x_4} \equiv0\pmod{I_4}.
\end{multline*}
Similarly, it is not hard to check that
\[
\left(x_2 x_4\right)R_{x_3}R_{x_1},\; \left(x_3
x_4\right)R_{x_1}R_{x_2},\; \left(x_3 x_4\right)R_{x_2}R_{x_1}\in
I_4.
\]
Thus, $I_4=\mathcal P_4\left(\V\right)$.

Further, let $u$ be an arbitrary monomial of $\pnv$ for
$n\geqslant5$. Then applying~\eqref{eq wSagle} one can order the
indices of variables of $u$ as follows:
\[
u=\left(x_{i_1}x_{i_2}\right)R_{x_{i_3}}\dots R_{x_{i_n}}, \quad
i_3<i_5,\quad i_4<\dots <i_n.
\]
Thus, similarly to the case $n=4$, combining~\eqref{eq Sagle} with
anticommutativity and~\eqref{eq wSagle}, one can obtain $u\in
I_n$.
\end{proof}

\begin{lemma}\label{lemma-UpperBoundMalc}
Any nontrivial linear combination of  basis words of $\pnv$ is not
an identity of ${\mathrm G}\left(\mathcal A\right)$.
\end{lemma}

\begin{proof}
Consider first the case $n=3$. We set
\[
f=\alpha\left(x_1 x_2\right)x_3+\beta\left(x_2
x_3\right)x_1+\gamma\left(x_3 x_1\right)x_2,\quad
\alpha,\beta,\gamma\in F
\]
and assume that $\Tilde{f}=0$ in $\A$. Then by the substitution
$x_1=a$, $x_2=x_3=y$, we have
\[
\Tilde{f}=\alpha (a\cdot y)\cdot y -\gamma (y\cdot a)\cdot y=
(\alpha+\gamma)a=0.
\]
Similarly, after two more substitutions $x_2=a$, $x_1=x_3=y$ and
$x_3=a$, $x_1=x_2=y$, we get the system
\[
\left\{
\begin{aligned}
\alpha+\gamma&=0,\\
\alpha+\beta&=0,\\
\beta+\gamma&=0.
\end{aligned}\right.
\]
The unique solution of this system in $F$ is
$\alpha=\beta=\gamma=0$.

Now consider a linear combination
\[
f_n=g_n+h_n+p_n+q_n
\]
of basis words of $\pnv$ for $n\geqslant 4$, where
\begin{align*}
g_n&=\alpha_1\left(x_1 x_2\right)R^{n-2}+\alpha_2\left(x_2 x_3\right)R^{n-2}+\alpha_3\left(x_3 x_1\right)R^{n-2},\\
h_n&=\sum_{i=4}^n \alpha_i\left(x_1 x_i\right)R^{n-2},\\
p_n&=\beta_1\left(x_1 x_2\right)\left(R_{x_3}\circ
R_{x_4}\right)R^{n-4} +\beta_2\left(x_2
x_3\right)\left(R_{x_1}\circ R_{x_4}\right)R^{n-4}
+\beta_3\left(x_3 x_1\right)\left(R_{x_2}\circ R_{x_4}\right)R^{n-4},\\
q_n&=\sum_{i=4}^n \beta_i\left(x_1 x_i\right)\left(R_{x_2}\circ
R_{x_3}\right)R^{n-4}, \quad \alpha_i,\beta_i\in F.
\end{align*}
Suppose that $\A$ satisfies the superidentity $\Tilde{f}_n=0$.
Then let us show that all the coefficients in $f_n$ are zero.

First we fix $i\geqslant 4$ and make in $\Tilde{f}_n=0$ the
substitution $x_i=a,\ x_j=y$ for all $j\neq i$. Then we get
\[
\alpha_i(y\cdot a)R^{n-2}_y= \left\{
\begin{aligned}
-\alpha_i v,\quad&\text{if $n$ is even},\\
-\alpha_i a,\quad&\text{if $n$ is odd},
\end{aligned}\right.
\]
whence, $\alpha_i=0$ for $i\geqslant 4$ and
\[
f_n=g_n+p_n+q_n.
\]
Similarly, for $i\geqslant 4$, by the substitution $x_i=w,\ x_j=e$
for all $j\neq i$, one can prove that $\beta_i=0$. Thus,
\[
f_n=g_n+p_n.
\]
Further, for $i=1,2,3$, by the substitution $x_i=a,\ x_j=y$ for
all $j\neq i$, similarly to the case $n=3$, we obtain
$\alpha_i=0$. Consequently,
\[
f_n=p_n.
\]
Finally, for $i=1,2,3$, by the substitution $x_i=w,\ x_j=e$ for
all $j\neq i$, we get $\beta_i=0$.
\end{proof}

\smallskip

By Lemmas~\ref{lemma-A-Malc2super}--\ref{lemma-UpperBoundMalc}, we
have $\V_{1,1}=\Tilde{\V}$. Thus, to complete the proof of
Theorem~\ref{theorem Malc2}, it remains to show that the chains
\[
\V_1\subset\V_2\subset\dots\subset\V_r\subset\dots\subset\V,\quad\;
\V_{0,1}\subset\V_{0,2}\subset\dots\subset\V_{0,s}\subset\dots\subset\Tilde{\V}
\]
ascend strictly.

\subsection{Strictness of inclusions $\V_{n}\subset\V_{n+1}$}%
\label{SubSec:InclusionsMalcLine}

Let $\mathrm G_n$ be the Grassmann algebra with the unit
$\mathbf{1}$ on the set $\mathbf{e}_1,\dots,\mathbf{e}_n$ of
anticommuting generators and $U_n$ be the algebra on the set
$e_1,\dots,e_n$ of generators with null multiplication. Consider a
split null extension $A_n=U_{n-1}\dotplus\mathrm G_{n-1}$ with an
anticommutative multiplication induced by the actions
\[
(\mathbf{e}_i\dots\mathbf{e}_j)\cdot
e_k=\mathbf{e}_i\dots\mathbf{e}_j\mathbf{e}_k.
\]

\begin{lemma}\label{lemma-An-Malc2n1}
$A_n$ is a $\V_n\text{-algebra}$.
\end{lemma}

\begin{proof}
By definition, $A_n$ is metabelian and anticommutative. Moreover,
it is not hard to see that $A_n$ can be generated by the elements
$\mathbf{1},e_1,\dots,e_{n-1}$. Thus it remains to check that
$A_n$ satisfies~\eqref{eq Sagle}. Indeed, for arbitrary
$w\in\mathrm G_n$, we have
\[
(w\cdot e_i)R_{e_j}R_{e_k}+(e_k\cdot w)R_{e_i}R_{e_j}=
w\mathbf{e}_i\mathbf{e}_j\mathbf{e}_k-w\mathbf{e}_k\mathbf{e}_i\mathbf{e}_j=0.
\]
\end{proof}

\medskip

Further, let us denote by $f_n=f_n\left(x_1,\dots,x_{n}\right)$
the polynomial
\[
f_n=\sum_{\sigma\in {\mathrm S}_n}{(-1)}^{|\sigma|}
\left(x_{\sigma(1)}x_{\sigma(2)}\right)R_{x_{\sigma(3)}}\dots
R_{x_{\sigma(n)}}.
\]

\begin{lemma}\label{lemma identVn}
The free algebra $\svob{\V}{X_n}$ satisfies the identity
$f_{n+1}=0$.
\end{lemma}

\begin{proof}
By definition, the value of $f_n$ is zero when values of two of
its variables coincide. Hence, by virtue of metability, it remains
to verify that $f_{n+1}$ takes zero value after a substitution
$x_{n+1}=w$, where $w=(x_ix_j)\tau$ for some operator word $\tau$.
Indeed, in view of anticommutativity and  metability, we have
\[
f_{n+1}(x_1,\dots,x_n,w) ={(-1)}^n2w\chi_n,\quad\;
\chi_n=\sum_{\sigma\in {\mathrm S}_n}{(-1)}^{|\sigma|}
R_{x_{\sigma(1)}}\dots R_{x_{\sigma(n)}}.
\]
Thus, on one hand, the action of $\chi_n$ is skew-symmetric with
respect to any pair of its variables by definition. On the other
hand, by virtue of~\eqref{eq SagleGen}, $\chi_n$ acts at $w$
symmetrically w.r.t. $x_i,x_j$. Therefore, this action is zero.
\end{proof}

In view of Lemmas~\ref{lemma-An-Malc2n1} and~\ref{lemma identVn},
to prove the strict inclusions $\V_{n}\subset\V_{n+1}$ it suffices
to check that $f_{n+1}$ takes a nonzero value in $A_{n+1}$.
Indeed,
\[
f_{n+1}\left(\mathbf{1},e_1,\dots,e_n\right) =2\sum_{\sigma\in
{\mathrm S}_n}{(-1)}^{|\sigma|} \mathbf{1}\,R_{e_{\sigma(1)}}\dots
R_{e_{\sigma(n)}} =2n!\,\mathbf{e}_1\dots\mathbf{e}_n\neq0.
\]

\begin{remark}
{\upshape V.~T.~Filippov~\cite{Filippov81} proved the strict
inclusion $\mathrm{Malc}_{n}\subset\mathrm{Malc}_{n+1}$ for all
$n\neq3$ and suggested the hypothesis
$\mathrm{Malc}_3=\mathrm{Malc}_4$. This hypothesis is still known
as a difficult open problem.}
\end{remark}

\subsection{Strictness of inclusions $\V_{0,n}\subset\V_{0,n+1}$}%
\label{SubSec:InclusionsMalcColumn}

\begin{proposition}\label{proposition RelMalcSvobSuper}
The free $\V\text{-superalgebra}$ $\svobs{\V}{Y}$ on an arbitrary
set $Y$ of odd generators satisfies the relations
\begin{align}
wR_xR_yR_z&=wR_zR_xR_y,\label{eq SupwSagle}\\
(xy)\rho R_x\eta R_y&=(xy)\rho R_y\eta R_x,\label{eq SupSagleGen22}\\
x^2\rho R_x\eta R_y&=x^2\rho R_y\eta R_x,\label{eq SupSagleGen31}
\end{align}
where $x,y,z\in Y$, $w\in\bigl(\svobs{\V}{Y}\bigr)^2$, and
$\rho,\eta$ are arbitrary operator words of the form $R_{y_i}\dots
R_{y_j}$, $y_i,\dots, y_j\in Y$.
\end{proposition}

\begin{proof}
Superizing~\eqref{eq Sagle} and taking into account metability and
superanticommutativity, we have
\begin{align*}
wR_xR_yR_z+{(-1)}^{|w|}wL_zR_xR_y=wR_xR_yR_z-wR_zR_xR_y&=0,\\
(xy)R_xR_y-(yx)R_yR_x+(xy)R_xR_y-(yx)R_yR_x=2(xy)\left[R_x,R_y\right]&=0,\\
x^2R_xR_y-x^2R_yR_x+(xy)R^2_x-(yx)R^2_x=x^2\left[R_x,R_y\right]&=0.
\end{align*}
Thus,~\eqref{eq SupwSagle} is proved and to conclude the proof
of~\eqref{eq SupSagleGen22},~\eqref{eq SupSagleGen31} it remains
to combine~\eqref{eq SupwSagle} with the last two obtained
relations.
\end{proof}

\smallskip

Further, let us denote by
$g_n=g_n\left(a,b,x_1,\dots,x_{n}\right)$ the polynomial
\[
g_n=(ab)\sum_{\sigma\in {\mathrm S}_n}R_{x_{\sigma(1)}}\dots
R_{x_{\sigma(n)}}.
\]

\begin{lemma}\label{lemma SupIdentV0n}
The free $\V\text{-superalgebra}$ $\svobs{\V}{Y_n}$ on a finite
set $Y_n=\{y_1,\dots,y_n\}$ of odd generators satisfies the
superidentity $\Tilde{g}_n=0$.
\end{lemma}

\begin{proof}
By definition, the value of $\Tilde{g}_n$ is zero when values of
at least two of its variables $x_1,\dots,x_n$ coincide. Hence, by
virtue of metability, it suffices to consider the case when
$\Tilde{g}_n$ takes a value
\[
\Tilde{g}_n=(ab)\chi_n,\quad\; \chi_n=\sum_{\sigma\in {\mathrm
S}_n}{(-1)}^{|\sigma|} R_{y_{\sigma(1)}}\dots R_{y_{\sigma(n)}}
\]
for some $a,b\in\svobs{\V}{Y_n}$. Thus, on one hand, the action of
$\chi_n$ is skew-symmetric with respect to any pair of its
variables by definition. On the other hand, by virtue of~\eqref{eq
SupSagleGen22} and~\eqref{eq SupSagleGen31}, $\chi_n$ acts at $ab$
symmetrically with respect to some pair of its variables.
Therefore, this action is zero.
\end{proof}

To conclude the proof of strict inclusions
$\V_{0,n}\subset\V_{0,n+1}$ let us construct a
$\V\text{-superalgebra}$ on a set of $n+1$ odd generators that
does not satisfy the superidentity $\Tilde{g}_n=0$.

\smallskip

Let $\mathrm G^{(n)}$ be the Grassmann algebra with the unit
$\mathbf{1}$ on the set $e_1,\dots,e_n$ of anticommuting
generators and $\Bar{\mathrm G}^{(n)}$ be the Grassmann algebra
without unit on the set $\Bar{e}_1,\dots,\Bar{e}_n$ of
anticommuting generators. For an element $w\in\mathrm G^{(n)}$ of
the form $w=e_{i_1}\dots e_{i_k}\neq\mathbf{1}$ we use the
notations $|w|=k$, $\Bar{w}=\Bar{e}_{i_1}\dots \Bar{e}_{i_k}$, and
set $|\mathbf{1}|=0$, $\Bar{\mathbf{1}}=0$. For $i=0,1$, by
$\mathrm G^{(n)}_i$ and $\Bar{\mathrm G}^{(n)}_i$ we denote,
respectively, the subspaces of $\mathrm G^{(n)}$ and $\Bar{\mathrm
G}^{(n)}$ spanned by all the words $w$ and $\Bar{w}$ such that
$|w|\equiv i \pmod 2$. Consider the direct some $M^{(n)}=\mathrm
G^{(n)}+\Bar{\mathrm G}^{(n)}$ of vector spaces, where we set
$M^{(n)}_0=\mathrm G^{(n)}_0+\Bar{\mathrm G}^{(n)}_1$ to be an
even component of $M^{(n)}$ and $M^{(n)}_1=\mathrm
G^{(n)}_1+\Bar{\mathrm G}^{(n)}_0$ to be its odd component.

Let $U^{(n)}$ be the superalgebra on the set $y_1,\dots,y_n$ of
odd generators with null multiplication. Consider a split null
extension $\A^{(n)}=U^{(n)}\dotplus M^{(n)}$ with a
superanticommutative multiplication induced by the actions
\[
w\cdot y_k=we_k, \quad \Bar{w}\cdot y_k=\Bar{w}\Bar{e}_k, \quad
w\in\mathrm G^{(n)}.
\]

By virtue of skew-symmetry of the elements of ${\mathrm G}^{(n)}$
and $\Bar{\mathrm G}^{(n)}$ with respect to their generators it is
not hard to prove the following

\begin{lemma}\label{lemma-An-Malc2super}
$\A^{(n)}$ is a $\V\text{-superalgebra}$.
\end{lemma}

\smallskip

Further, we consider an extension
$\Bar{\A}^{(n+1)}=F\cdot x+\A^{(n)}$
of the superalgebra
$\A^{(n)}$
such that
\[
\Bar{\A}^{(n+1)}_0=\A^{(n)}_0, \quad\; \Bar{\A}^{(n+1)}_1=F\cdot x
+\A^{(n)}_1,
\]
and all nonzero products of basis elements with~$x$
are the following:
\[
x^2=\mathbf{1},
\quad\;
x\cdot y_i=y_i\cdot x=\Bar{e}_i,
\quad\;
 x\cdot \Bar{w}
={(-1)}^{|w|}\Bar{w}\cdot x
=\frac{1}{2}w, \quad \Bar{w}\in \Bar{\mathrm G}^{(n)}.
\]

\begin{lemma}\label{lemma-An+1-Malc2super}
$\Bar{\A}^{(n+1)}$ is a $\V\text{-superalgebra}$ generated by
$n+1$ odd element.
\end{lemma}

\begin{proof}
By Lemma~\ref{lemma-An-Malc2super}, taking into account
skew-symmetry of the elements of ${\mathrm G}^{(n)}$ and
$\Bar{\mathrm G}^{(n)}$ with respect to their generators, it
suffices to verify that the superization of~\eqref{eq Sagle} holds
in the following cases:
\begin{multline*}
x^2R_{y_i}R_{y_j}-(x\cdot y_i)R_{y_j}R_x-(y_j\cdot x)R_xR_{y_i}
=(\mathbf{1}\cdot y_i)\cdot y_j-(\Bar{e}_i\cdot y_j)\cdot x-(\Bar{e}_j\cdot x)\cdot y_i=\\
=e_i\cdot y_j-\Bar{e}_i\Bar{e}_j\cdot x+\frac{1}{2}e_j\cdot y_i
=e_ie_j-\frac{1}{2}e_ie_j+\frac{1}{2}e_je_i=0,
\end{multline*}
\begin{multline*}
(x\cdot y_i)R_xR_{y_j}-(y_i\cdot x)R_{y_j}R_x+(x\cdot y_j)R_xR_{y_i}-(y_j\cdot x)R_{y_i}R_x=\\
=(\Bar{e}_i\cdot x)\cdot y_j-(\Bar{e}_i\cdot y_j)\cdot x+(\Bar{e}_j\cdot x)\cdot y_i-(\Bar{e}_j\cdot y_i)\cdot x=\\
=-\frac{1}{2}e_i\cdot y_j-\Bar{e}_i\Bar{e}_j\cdot
x-\frac{1}{2}e_j\cdot y_i-\Bar{e}_j\Bar{e}_i\cdot x
=-\frac{1}{2}e_ie_j-\frac{1}{2}e_ie_j-\frac{1}{2}e_je_i-\frac{1}{2}e_je_i=0,
\end{multline*}
\begin{multline*}
(x\cdot\Bar{w})R_{y_i}R_{y_j}-{(-1)}^{|w|}(\Bar{w}\cdot y_i)R_{y_j}R_x=\\
=\frac{1}{2}(w\cdot y_i)\cdot y_j-{(-1)}^{|w|}\Bar{w}\Bar{e}_i\Bar{e}_j\cdot x
=\frac{1}{2}we_ie_j-\frac{1}{2}we_ie_j=0,
\end{multline*}
\begin{multline*}
(\Bar{w}\cdot x)R_{y_i}R_{y_j}-{(-1)}^{|w|}(y_j\cdot \Bar{w})R_xR_{y_i}
=\frac{{(-1)}^{|w|}}{2}(w\cdot y_i)\cdot y_j-(\Bar{w}\Bar{e}_j\cdot x)\cdot y_i=\\
=\frac{{(-1)}^{|w|}}{2}we_ie_j-\frac{{(-1)}^{|w|+1}}{2}we_j\cdot y_i
=\frac{{(-1)}^{|w|}}{2}\left(we_ie_j+we_je_i\right)=0,
\end{multline*}
\begin{multline*}
(\Bar{w}\cdot y_i)R_xR_{y_j}-{(-1)}^{|w|}(y_j\cdot \Bar{w})R_{y_i}R_x
=(\Bar{w}\bar{e}_i\cdot x)\cdot y_j-(\Bar{w}\bar{e}_j\cdot y_i)\cdot x=\\
=\frac{{(-1)}^{|w|+1}}{2}we_i\cdot y_j-\Bar{w}\Bar{e}_j\Bar{e_i}\cdot x
=-\frac{{(-1)}^{|w|}}{2}\left(we_ie_j+we_je_i\right)=0.
\end{multline*}
\end{proof}

In view of Lemmas~\ref{lemma
SupIdentV0n}--\ref{lemma-An+1-Malc2super}, to prove the strict
inclusions $\V_{0,n}\subset\V_{0,n+1}$ it remains to check that
$\Tilde{g}_n$ takes a nonzero value in $\Bar{\A}^{(n+1)}$. Indeed,
\[
\Tilde{g}_n\left(x,x,y_1,\dots,y_n\right) =\sum_{\sigma\in
{\mathrm S}_n}{(-1)}^{|\sigma|} \mathbf{1}\,R_{y_{\sigma(1)}}\dots
R_{y_{\sigma(n)}} =n!\,e_1\dots e_n\neq0.
\]

Theorem~\ref{theorem Malc2} is proved.

\smallskip

\begin{remark}
{\upshape We stress that the description of the inclusions in the
lattice $\lt{\V}$ obtained in Theorem~\ref{theorem Malc2} is more
detailed than one needs to deduce the uniqueness of the basic
superrank $(1,1)$ for $\V$. Actually proving that
$\bsp{\V}=\bigl\{(1,1)\bigr\}$ we could restrict with establishing
the equality $\V_{1,1}=\Tilde{\V}$ and the strict inclusions
$\V_n\subset\V$, $\V_{0,n}\subset\Tilde{\V}$. Namely, in view of
Lemmas~\ref{lemma-A-Malc2super},~\ref{lemma identVn},
and~\ref{lemma SupIdentV0n}, it is enough to check that the
superpolynomials $\Tilde{f_n}$ $\Tilde{g_n}$ don't vanish on some
elements of the superalgebra $\A$.}
\end{remark}

\smallskip

\begin{remark}
{\upshape
It is not hard to check that in fact the superalgebra
$\Bar{\A}^{(n+1)}$
is a Lie superalgebra.
The variety
$\mathrm{Lie}^{(2)}$
of metabelian Lie algebras has a finite basic rank:
it is generated by a 2--dimensional non-abelian algebra and hence
$\mathit{r_{b}}\bigl(\mathrm{Lie}^{(2)}\bigr)=2$.
Since we are interested in metabelian varieties of infinite basic rank,
it was not planed to consider the variety
$\mathrm{Lie}^{(2)}$.
But in view of the results of this section we notice that the lattices
$\ltt{\mathrm{Lie}^{(2)}}$
and
$\ltt{\mathrm{Malc}^{(2)}}$
have a quite similar structure in spite of the difference in their initial chains.
Indeed, one can easily prove that
$\mathrm{Lie}^{(2)}$
also
possesses the basic superrank~$(1,1)$.
Moreover, the strict inclusions
$\mathrm{Lie}^{(2)}_{\,0,n}\subset\mathrm{Lie}^{(2)}_{\,0,n+1}$
are provided by the fact that the free metabelian Lie superalgebra on~$n$
odd generators is nilpotent of index~$n+2$.
Thus the lattice
$\lt{\V}$,
for
$\V=\mathrm{Lie}^{(2)}$,
has the form
\[
\begin{matrix}
            &          & \mathcal V_{1,0}   & \subset  & \mathcal V_{2,0}   & =\cdots= & \mathcal V_{r,0}   & =  \cdots   \\
            &          &  \cap      &          & \shortmid\shortmid      &                      & \shortmid\shortmid       &        \\
 \mathcal V_{0,1}   & \subset  & \mathcal V_{1,1}   & =  & \mathcal V_{2,1}   & =\cdots= & \mathcal V_{r,1}   & =\cdots   \\
 \cap &          &  \shortmid\shortmid      &          & \shortmid\shortmid       &                      & \shortmid\shortmid       &         \\
 \mathcal V_{0,2}   & \subset & \mathcal V_{1,2}   &=  & \mathcal V_{2,2}   & =\cdots= & \mathcal V_{r,2}   & =\cdots  \\
  \cap      &          &  \shortmid\shortmid      &          & \shortmid\shortmid       &                      & \shortmid\shortmid       &         \\
 \mathcal V_{0,3}   & \subset & \mathcal V_{1,3}   &=  & \mathcal V_{2,3}   & =\cdots= & \mathcal V_{r,3}   & =\cdots  \\
   \cap       &          &  \shortmid\shortmid      &          & \shortmid\shortmid       &                      & \shortmid\shortmid       &         \\
\hdotsfor{8}
\end{matrix}
\]
In particular, $\bspp{\mathrm{Lie}^{(2)}}=\bigl\{(2,0),(1,1)\bigr\}$.}
\end{remark}

\section{Metabelian algebras}%
\label{Sec:MetabelianAlgebras}

\subsection{Endomorphisms corresponding to Young tables}%
\label{SubSec:EndomorphismsCorrespondingToYoungTables}

Let $d$ be a \textit{Young diagram}
of order~$n$ and $\tau$ be a permutation of ${\mathrm S}_{n}$.
Following~\cite[Chap.~3]{Bahturin87}, by $\tau d$ we denote the
\textit{Young table} obtained by filing in the diagram~$d$ with
the numbers $\tau(1),\dots,\tau(n)$ in the order from the top down
and from left to right. The group ${\mathrm S}_{n}$ acts naturally
on the set of all tables corresponding to~$d$: $\sigma(\tau
d)=(\sigma\tau)d$. By ${\mathrm C}_{\tau d}$ we denote the
\textit{column stabilizer of the table~$\tau d$}, i.~e. the
subgroup of ${\mathrm S}_{n}$ consisting of all permutations
preserving the set of symbols in each column of~$\tau d$.
Similarly, one can define the \textit{row stabilizer~${\mathrm
R}_{\tau d}$  of~$\tau d$}.

By $\mathcal P_n$ we denote the \textit{linear span of all
associative multilinear words of degree~$n$ on the set~$X_n$ of
variables.} For every Young table $\tau d$ of order $n'\leqslant
n$ we define the following endomorphisms:
\begin{gather*}
\varphi_{\tau d},\psi_{\tau d}:\mathcal P_n\mapsto \mathcal P_n,\\
\varphi_{\tau d}\left(w\right) =\sum_{\sigma\in{\mathrm C}_{\tau
d}} \sum_{\rho\in{\mathrm R}_{\sigma\tau d}}
{(-1)}^{\left|\sigma\right|}
w\left(x_{\rho\sigma(1)},\dots,x_{\rho\sigma(n')},x_{n'+1},\ldots,x_n\right),\\
\psi_{\tau d}\left(w\right) =\sum_{\rho\in{\mathrm R}_{\tau d}}
\sum_{\sigma\in{\mathrm C}_{\rho\tau d}}
{(-1)}^{\left|\sigma\right|}
w\left(x_{\sigma\rho(1)},\dots,x_{\sigma\rho(n')},x_{n'+1},\ldots,x_n\right).
\end{gather*}

For example, if $ \tau d= \lefteqn{ \phantom{
\begin{aligned}
\vline\hspace{5pt}
\begin{matrix}
\hline
1      \\
\hline
3      \\
\hline
\end{matrix}
\vline
\end{aligned}\;}
\vline}
\begin{aligned}
\vline\hspace{5pt}
\begin{matrix}
\hline
1 & 2 \\
\hline
3 & 4 \\
\hline
\end{matrix}
\hspace{5pt}\vline
\end{aligned}\:
$ and $w=x_1x_2x_3x_4$, then $\varphi_{\tau d}\left(w\right)$ can
be calculated as follows:
\begin{multline*}
\varphi_{\tau d}(w)= \sum_{\rho\in{\mathrm R}_{\tau d}}
x_{\rho(1)}x_{\rho(2)}x_{\rho(3)}x_{\rho(4)}
-\sum_{\rho\in{\mathrm R}_{(13)\tau d}}
x_{\rho(3)}x_{\rho(2)}x_{\rho(1)}x_{\rho(4)}-\\
-\sum_{\rho\in{\mathrm R}_{(24)\tau d}}
x_{\rho(1)}x_{\rho(4)}x_{\rho(3)}x_{\rho(2)}
+\sum_{\rho\in{\mathrm R}_{(13)(24)\tau d}}
x_{\rho(3)}x_{\rho(4)}x_{\rho(1)}x_{\rho(2)}
\end{multline*}
and, taking into account that
\[
(13)\tau d= \lefteqn{ \phantom{
\begin{aligned}
\vline\hspace{5pt}
\begin{matrix}
\hline 3      \\
\hline 1      \\
\hline
\end{matrix}
\vline
\end{aligned}\;}
\vline}
\begin{aligned}
\vline\hspace{5pt}
\begin{matrix}
\hline
3      & 2 \\
\hline
1      & 4 \\
\hline
\end{matrix}
\hspace{5pt}\vline
\end{aligned}
\:,\qquad (24)\tau d= \lefteqn{ \phantom{
\begin{aligned}
\vline\hspace{5pt}
\begin{matrix}
\hline 1      \\
\hline 3      \\
\hline
\end{matrix}
\vline
\end{aligned}\;}
\vline}
\begin{aligned}
\vline\hspace{5pt}
\begin{matrix}
\hline
1      & 4 \\
\hline
3      & 2 \\
\hline
\end{matrix}
\hspace{5pt}\vline
\end{aligned}
\:,\qquad (13)(24)\tau d= \lefteqn{ \phantom{
\begin{aligned}
\vline\hspace{5pt}
\begin{matrix}
\hline 3      \\
\hline 1      \\
\hline
\end{matrix}
\vline
\end{aligned}\;}
\vline}
\begin{aligned}
\vline\hspace{5pt}
\begin{matrix}
\hline
3      & 4 \\
\hline
1      & 2 \\
\hline
\end{matrix}
\hspace{5pt}\vline
\end{aligned}\:,
\]
we get
\begin{multline*}
\varphi_{\tau d}(w)= (x_1\circ x_2)(x_3\circ x_4)
-(x_3\circ x_2)(x_1\circ x_4)-\\
-(x_1\circ x_4)(x_3\circ x_2)
+(x_3\circ x_4)(x_1\circ x_2)=\\
=(x_1\circ x_2)\circ(x_3\circ x_4)-(x_3\circ x_2)\circ(x_1\circ
x_4).
\end{multline*}

Similarly, for the same table $\tau d$ and $w'=x_1x_3x_2x_4$, we
have
\begin{multline*}
\psi_{\tau d}\left(w'\right) =\sum_{\sigma\in{\mathrm C}_{\tau d}}
{(-1)}^{\left|\sigma\right|}
x_{\sigma(1)}x_{\sigma(3)}x_{\sigma(2)}x_{\sigma(4)}
+\sum_{\sigma\in{\mathrm C}_{(12)\tau d}}
{(-1)}^{\left|\sigma\right|}
x_{\sigma(2)}x_{\sigma(3)}x_{\sigma(1)}x_{\sigma(4)}+\\
+\sum_{\sigma\in{\mathrm C}_{(34)\tau d}}
{(-1)}^{\left|\sigma\right|}
x_{\sigma(1)}x_{\sigma(4)}x_{\sigma(2)}x_{\sigma(3)}
+\sum_{\sigma\in{\mathrm C}_{(12)(34)\tau d}}
{(-1)}^{\left|\sigma\right|}
x_{\sigma(2)}x_{\sigma(4)}x_{\sigma(1)}x_{\sigma(3)}
\end{multline*}
and, observing that
\[
(12)\tau d= \lefteqn{ \phantom{
\begin{aligned}
\vline\hspace{5pt}
\begin{matrix}
\hline 2      \\
\hline 3      \\
\hline
\end{matrix}
\vline
\end{aligned}\;}
\vline}
\begin{aligned}
\vline\hspace{5pt}
\begin{matrix}
\hline
2      & 1 \\
\hline
3      & 4 \\
\hline
\end{matrix}
\hspace{5pt}\vline
\end{aligned}
\:,\qquad (34)\tau d= \lefteqn{ \phantom{
\begin{aligned}
\vline\hspace{5pt}
\begin{matrix}
\hline 1      \\
\hline 4      \\
\hline
\end{matrix}
\vline
\end{aligned}\;}
\vline}
\begin{aligned}
\vline\hspace{5pt}
\begin{matrix}
\hline
1      & 2 \\
\hline
4      & 3 \\
\hline
\end{matrix}
\hspace{5pt}\vline
\end{aligned}
\:,\qquad (12)(34)\tau d= \lefteqn{ \phantom{
\begin{aligned}
\vline\hspace{5pt}
\begin{matrix}
\hline 2      \\
\hline 4      \\
\hline
\end{matrix}
\vline
\end{aligned}\;}
\vline}
\begin{aligned}
\vline\hspace{5pt}
\begin{matrix}
\hline
2      & 1 \\
\hline
4      & 3 \\
\hline
\end{matrix}
\hspace{5pt}\vline
\end{aligned}
\:,
\]
we obtain
\begin{multline*}
\psi_{\tau d}\left(w'\right)
=[x_1,x_3][x_2,x_4]+[x_2,x_3][x_1,x_4]
+[x_1,x_4][x_2,x_3]+[x_2,x_4][x_1,x_3]=\\
=[x_1,x_3]\circ[x_2,x_4]+[x_2,x_3]\circ[x_1,x_4].
\end{multline*}

We stress that, by definition, the polynomial $\varphi_{\tau
d}\left(w\right)$ is skew-symmetric with respect to each pair of
its variables whose indices stand at the same column of $\tau d$.
Similarly, $\psi_{\tau d}\left(w\right)$ is symmetric with respect
to each pair of its variables whose indices stand at the same row
of $\tau d$.

In what follows, the introduced endomorphisms are considered for
rectangular Young tables only. The notation $k\times m$ means that
a table has $k$ rows and $m$ columns.

By $\mathfrak A_{r,s}$ denote the \textit{free associative
superalgebra on the set $Y_{r,s}$ of~$r$ even and~$s$ odd
generators}.

\begin{lemma}\label{lemma Endphi}
Let $\xi:\mathcal P_n\mapsto \mathfrak A_{r,s}$ be a homomorphism
induced by a mapping $\xi:X_n\mapsto Y_{r,s}$, where
$n\geqslant(r+1)(rs+s+1)$. Then for every rectangular $(r+1)\times
(rs+s+1)$ table $\tau d$ and every word $w\in\mathcal P_n$, the
superization of the polynomial $\xi\left(\varphi_{\tau
d}\left(w\right)\right)$ is equal to zero.
\end{lemma}

\begin{proof}
By definition of the polynomial $f=\varphi_{\tau
d}\left(w\right)$, we have $f=\sum_{\sigma\in{\mathrm C}_{\tau
d}}{(-1)}^{\left|\sigma\right|}f'_{\sigma}$, where
\[
f'_{\sigma}=\sum_{\rho\in{\mathrm R}_{\sigma\tau d}}
w\left(x_{\rho\sigma(1)},\dots,x_{\rho\sigma(n')},x_{n'+1},\ldots,x_n\right),\quad
n'=(r+1)(rs+s+1).
\]
Let us substitute the element $i$ by $\xi(x_i)$ in each cell of
the table $\tau d$. The obtained table we denote by ${(\tau
d)}^\xi$.

First, in view of skew-symmetry of $f$ w.r.t. each pair of its
variables whose indices stand at the same column of $\tau d$, we
note that $\widetilde{\xi(f)}$ can be non-zero only if none of the
columns of ${(\tau d)}^\xi$ contains the same even generator of
$\mathfrak A_{r,s}$ twice. Consequently, the whole number of even
elements in ${(\tau d)}^\xi$ is not more then $r(rs+s+1)$.

Further, if $\widetilde{\xi(f)}\neq0$, then we may assume that
$\widetilde{\xi(f'_{\sigma})}\neq0$ at least for one
$\sigma\in{\mathrm C}_{\tau d}$. In this case, we stress that
$f'_{\sigma}$ is symmetric w.r.t. each pair of its variables whose
indices stand at the same row of $\sigma\tau d$. Hence, none of
the rows of ${(\sigma\tau d)}^\xi$ contains the same odd element
twice. Consequently, the whole number of odd elements in ${(\tau
d)}^\xi$ is not more then $(r+1)s$.

Therefore, the number of elements in ${(\tau d)}^\xi$ is not more
then
\[
r(rs+s+1)+(r+1)s=(r+1)(rs+s+1)-1.
\]
The obtained contradiction completes the proof.
\end{proof}

By the similar arguments one can prove the following

\begin{lemma}\label{lemma Endpsi}
Let $\xi:\mathcal P_n\mapsto \mathfrak A_{r,s}$ be a homomorphism
induced by a mapping $\xi:X_n\mapsto Y_{r,s}$, where
$n\geqslant(rs+r+1)(s+1)$. Then for every rectangular
$(rs+r+1)\times (s+1)$ table $\tau d$ and every word $w\in\mathcal
P_n$, the superization of $\xi\left(\psi_{\tau
d}\left(w\right)\right)$ is equal to zero.
\end{lemma}

\subsection{Inclusions in the lattice $\lt{\mathfrak M}$}%
\label{SubSec:InclusionsInTheLatticeLM}

\begin{lemma}\label{lemma InclR}
For all naturals $r,s$, we have $\mathfrak
M_{r}\nsubseteq\mathfrak M_{r-1,s}$.
\end{lemma}

\begin{proof}
Let $f_k=f_k\left(u,v,x_1,\dots,x_{kr}\right)$ be the polynomial
\[
f_k=\left(uv\right) \sum_{\sigma\in{\mathrm C}_{\tau d}}
\sum_{\rho\in{\mathrm R}_{\sigma\tau d}}
{(-1)}^{\left|\sigma\right|}R_{x_{\rho\sigma(1)}}\dots
R_{x_{\rho\sigma(kr)}},
\]
where
\[
\tau d= \lefteqn{ \phantom{ \vline\hspace{5pt}
\begin{matrix}
\hline
1      \\
\hline
2      \\
\hline
\vdots  \\
\hline
r     \\
\hline
\end{matrix}
\vline\; }\vline} \lefteqn{ \phantom{ \vline\hspace{5pt}
\begin{matrix}
\hline
1      & r+1 \\
\hline
2      & r+2 \\
\hline
\vdots  & \vdots  \\
\hline
r     & 2r  \\
\hline
\end{matrix}
\vline\; }\vline} \lefteqn{ \phantom{ \vline\hspace{5pt}
\begin{matrix}
\hline
1      & r+1 & \dots  \\
\hline
2      & r+2 & \dots  \\
\hline
\vdots  & \vdots  & \ddots \\
\hline
r     & 2r  & \dots  \\
\hline
\end{matrix}
\vline\; }\vline} \vline\hspace{5pt}
\begin{matrix}
\hline
1      & r+1 & \dots  & (k-1)r+1 \\
\hline
2      & r+2 & \dots  & (k-1)r+2\\
\hline
\vdots  & \vdots  & \ddots & \vdots \\
\hline
r     & 2r  & \dots  &  kr\\
\hline
\end{matrix}
\hspace{5pt}\vline \:.
\]
We stress that applying Lemma~\ref{lemma Endphi}, it is not hard
to prove that $\Tilde{f}_k=0$ is a superidentity in $\mathfrak
M_{r-1,s}$ for $k=rs+1$. Thus to prove the Lemma it suffices to
construct some $\mathfrak M_{r}\text{-algebra}$ $A^{(r)}$ such
that $f_k$ takes nonzero values in $A^{(r)}$ for all $k$.

Consider the algebra
\[
U^{(r)}=\sum_{i=1}^r F\cdot e_i
\]
with null multiplication and the vector space
\[
M^{(r)}=\sum_{n\in\mathbb Z_r}F\cdot a_{n}.
\]
Let $A^{(r)}=U^{(r)}\dotplus M^{(r)}$ be the split null extension
with the multiplication induced by the following actions:
\[
a_{n}\cdot e_i=\left\{
\begin{aligned}
a_{n+1\!\!\!\pmod{r}},\quad n&\equiv i \pmod{r},\\
0,                 \quad n&\lefteqn{\;\,\slash}\equiv i \pmod{r}.
\end{aligned}\right.
\]
By definition, $A^{(r)}$ is a metabelian algebra generated by the
elements $a_{\Bar0}+e_r,e_1,\dots,e_{r-1}$.

Consider the mapping $\xi:X_{kr}\mapsto
\left\{e_1,\dots,e_r\right\}$ defined by the $r\times k$ table
\[
{(\tau d)}^\xi= \lefteqn{ \phantom{ \vline\hspace{5pt}
\begin{matrix}
\hline
e_1      \\
\hline
e_2      \\
\hline
\vdots  \\
\hline
e_r      \\
\hline
\end{matrix}
\;\vline} \vline} \lefteqn{ \phantom{ \vline\hspace{5pt}
\begin{matrix}
\hline
e_1      & e_1 \\
\hline
e_2      & e_2 \\
\hline
\vdots  & \vdots  \\
\hline
e_r      & e_r  \\
\hline
\end{matrix}
\;\vline} \vline} \lefteqn{ \phantom{ \vline\hspace{5pt}
\begin{matrix}
\hline
e_1      & e_1 & \dots  \\
\hline
e_2      & e_2 & \dots  \\
\hline
\vdots  & \vdots  & \ddots \\
\hline
e_r      & e_r  & \dots  \\
\hline
\end{matrix}
\;\vline} \vline} \vline\hspace{5pt}
\begin{matrix}
\hline
e_1      & e_1 & \dots  & e_1 \\
\hline
e_2      & e_2 & \dots  & e_2 \\
\hline
\vdots  & \vdots  & \ddots & \vdots \\
\hline
e_r      & e_r  & \dots  & e_r\\
\hline
\end{matrix}
\hspace{5pt}\vline \:,
\]
i.~e. $\xi(x_i)$ is the element of the ${(\tau d)}^\xi$ standing
in the cell corresponding to the index~$i$ in $\tau d$. To
conclude the proof it remains to verify that $f_k$ takes a nonzero
value in $A^{(r)}$. Indeed,
\[
f_k\left(a_{\Bar0},e_r,\xi(x_1),\dots,\xi(x_{kr})\right)=k!\,r\,a_{\Bar1}\neq0.
\]
\end{proof}

\begin{lemma}\label{lemma InclS}
For all naturals $r,s$, we have $\mathfrak
M_{0,s}\nsubseteq\mathfrak M_{r,s-1}$.
\end{lemma}

\begin{proof}
Let $f_k=f_k\left(u,v,x_1,\dots,x_{ks}\right)$ be the polynomial
\[
f_k=\left(uv\right) \sum_{\rho\in{\mathrm R}_{\tau d}}
\sum_{\sigma\in{\mathrm C}_{\rho\tau d}}
{(-1)}^{\left|\sigma\right|}R_{x_{\sigma\rho(1)}}\dots
R_{x_{\sigma\rho(ks)}},
\]
where
\[
\tau d= \lefteqn{ \phantom{ \vline\hspace{5pt}
\begin{matrix}
\hline
1      \\
\hline
s+1      \\
\hline
\vdots  \\
\hline
(k-1)s+1     \\
\hline
\end{matrix}
\;\vline }\vline} \lefteqn{ \phantom{ \vline\hspace{5pt}
\begin{matrix}
\hline
1      & 2 \\
\hline
s+1      & s+2 \\
\hline
\vdots  & \vdots \\
\hline
(k-1)s+1      & (k-1)s+2 \\
\hline
\end{matrix}
\;\vline }\vline} \lefteqn{ \phantom{ \vline\hspace{5pt}
\begin{matrix}
\hline
1      & 2 & \dots  \\
\hline
s+1      & s+2 & \dots \\
\hline
\vdots  & \vdots  & \ddots  \\
\hline
(k-1)s+1      & (k-1)s+2  & \dots  \\
\hline
\end{matrix}
\;\vline }\vline} \vline\hspace{5pt}
\begin{matrix}
\hline
1      & 2 & \dots  & s \\
\hline
s+1      & s+2 & \dots  & 2s \\
\hline
\vdots  & \vdots  & \ddots & \vdots \\
\hline
(k-1)s+1      & (k-1)s+2  & \dots  &  ks\\
\hline
\end{matrix}
\hspace{5pt}\vline \:.
\]
We stress that applying Lemma~\ref{lemma Endpsi}, it is not hard
to prove that $\Tilde{f}_k=0$ is a superidentity in $\mathfrak
M_{r,s-1}$ for $k=rs+1$. Thus to prove the Lemma it suffices to
construct some $\mathfrak M_{0,s}\text{-superalgebra}$ $\A^{(s)}$
such that $\Tilde{f}_k$ takes nonzero values in $\A^{(s)}$ for all
$k$.

Let $U^{(s)}=U^{(s)}_0+U^{(s)}_1$ be the superalgebra
\[
U^{(s)}_0=\{0\},\quad\; U^{(s)}_1=\sum_{i=1}^s F\cdot y_i
\]
with null multiplication and $M^{(s)}=M^{(s)}_0+M^{(s)}_1$ be the
vector space
\[
M^{(s)}_0=\sum_{n\in\mathbb Z_s}F\cdot a_{2n},\quad\;
M^{(s)}_1=\sum_{n\in\mathbb Z_s}F\cdot a_{2n+1}.
\]
Consider a split null extension $\A^{(s)}=U^{(s)}\dotplus M^{(s)}$
with a multiplication induced by the following actions:
\[
a_{n}\cdot y_i=\left\{
\begin{aligned}
a_{n+1\!\!\!\pmod{2s}},\quad n&\equiv i \pmod{s},\\
0,                 \quad n&\lefteqn{\;\,\slash}\equiv i \pmod{s}.
\end{aligned}\right.
\]
One can easily check that $\A^{(s)}$ is a metabelian superalgebra
generated by the odd elements $a_{\Bar1}+y_1,y_2,\dots,y_s$.

Let us consider the mapping $\xi:X_{ks}\mapsto
\left\{y_1,\dots,y_s\right\}$ defined by the $k\times s$ table
\[
{(\tau d)}^\xi= \lefteqn{ \phantom{ \vline\hspace{5pt}
\begin{matrix}
\hline
y_1     \\
\hline
y_1     \\
\hline
\vdots  \\
\hline
y_1     \\
\hline
\end{matrix}
\;\vline }\vline} \lefteqn{ \phantom{ \vline\hspace{5pt}
\begin{matrix}
\hline
y_1      & y_2 \\
\hline
y_1      & y_2 \\
\hline
\vdots  & \vdots \\
\hline
y_1      & y_2  \\
\hline
\end{matrix}
\;\vline }\vline} \lefteqn{ \phantom{ \vline\hspace{5pt}
\begin{matrix}
\hline
y_1      & y_2 & \dots  \\
\hline
y_1      & y_2 & \dots  \\
\hline
\vdots  & \vdots  & \ddots \\
\hline
y_1      & y_2  & \dots  \\
\hline
\end{matrix}
\;\vline }\vline} \vline\hspace{5pt}
\begin{matrix}
\hline
y_1      & y_2 & \dots  & y_s \\
\hline
y_1      & y_2 & \dots  & y_s \\
\hline
\vdots  & \vdots  & \ddots & \vdots \\
\hline
y_1      & y_2  & \dots  &  y_s \\
\hline
\end{matrix}
\hspace{5pt}\vline \:,
\]
i.~e. $\xi(x_i)$ is the element of the ${(\tau d)}^\xi$ standing
in the cell corresponding to the index~$i$ in $\tau d$. To
conclude the proof it remains to verify that $\Tilde{f}_k$ takes a
nonzero value in $\A^{(s)}$. Indeed,
\[
\frac{1}{k!s}\Tilde{f}_k\left(a_{\Bar0},y_s,\xi(x_1),\dots,\xi(x_{ks})\right)
=\left\{
\begin{aligned}
a_{\Bar1},\quad&\text{if $k$ is even},\\
a_{\overline{s+1}},\quad&\text{if $k$ is odd}.
\end{aligned}\right.
\]
\end{proof}

\medskip

Lemmas~\ref{lemma InclR} and~\ref{lemma InclS} yield that
$\mathfrak M_{r',s}\nsubseteq\mathfrak M_{r,s'}$ and $\mathfrak
M_{r,s'}\nsubseteq\mathfrak M_{r',s}$ for all nonnegative integers
$r'<r$, $s'<s$. Moreover, the inclusions $\mathfrak
M_{r',s}\subset\mathfrak M_{r,s}$ and $\mathfrak
M_{r,s'}\subset\mathfrak M_{r,s}$ are strict.

\smallskip

Theorem~\ref{theorem MetabInclusions} is proved.

\smallskip

Corollary~\ref{corollary MetabInfin} is an immediate consequence
of Theorem~\ref{theorem MetabInclusions}.

\subsection{Variety of unique arbitrary given finite basic superrank}%
\label{SubSec:VarietyOfUniqueBSR}

Let us deduce Corollary~\ref{corollary MetabArbit}. For an
arbitrary pare $(r,s)\neq(0,0)$ of nonnegative integers, consider
the free $\mathfrak M_{r,s}\text{-superalgebra}$ $\mathfrak
A_{r,s}$. Let $\mathfrak N=\vr \mathrm G\left(\mathfrak
A_{r,s}\right)$ be the variety generated by the Grassmann envelope
of $\mathfrak A_{r,s}$. Then, by definition, we have
$\Tilde{\mathfrak N}=\mathfrak M_{r,s}$. It remains to prove that
$\mathfrak N$ doesn't possess any other basic superrank $(r',s')$
such that at least one of the inequalities $r'<r$, $s'<s$ holds.
Indeed, we stress that Theorem~\ref{theorem MetabInclusions}
states the strict inclusion
\[
\left(\mathfrak M_{r,s}\cap\mathfrak
M_{r',s'}\right)\subset\mathfrak M_{r,s}.
\]
At the same time, it is clear that $\mathfrak N_{r',s'}=\mathfrak
M_{r,s}\cap\mathfrak M_{r',s'}$. Thus, $\mathfrak N_{r',s'}\subset
\Tilde{\mathfrak N}$, i.~e.\linebreak $(r',s')\notin\bsp{\mathfrak
N}$. Therefore, $\bsp{\mathfrak N}=\left\{(r,s)\right\}$.

\section{Right alternative and right symmetric algebras}%
\label{Sec:RightAlternativeAndRightSymmetricAlgebras}

The variety of right alternative algebras is a well-known source
of examples of nonfinitely based
varieties~\cite{Belkin76,Isaev86,Kuz'min15,Pchelintsev07} over a
field of characteristic zero. The similar results can be obtained
for the variety of right symmetric algebras. In this section, we
construct the varieties of right alternative and right symmetric
algebras having no finite basic superrank.

\smallskip

Recall that by $\Ve$ $(\varepsilon=\pm1)$ we denote the subvariety
of $\mathfrak M$ distinguished by identities~\eqref{eq
EpsilonSymm} and~\eqref{eq EpsilonNil2}. Let us prove
Theorem~\ref{theorem VeInfin}.

\subsection{Auxiliary $\Ve\text{-superalgebra}$}%
\label{SubSec:AuxiliarySuperalgebra}

Let $U=U_0+U_1$ be the superalgebra
\[
U_0=\{0\},\quad\; U_1=\sum_{i=1}^\infty F\cdot y_i
\]
with null multiplication and $M=M_0+M_1$ be the vector space
\[
M_0=\sum_{i=1}^\infty F\cdot a_i,\quad\; M_1=\sum_{i=1}^\infty
F\cdot w_i.
\]
Consider a split null extension $\Ae=U\dotplus M$ such that all
nonzero products of the basis elements of $\Ae$ are the following:
\[
y_i\cdot a_i=\varepsilon\, a_i\cdot y_i=w_{i+1},\quad\; y_i\cdot
w_i=a_i.
\]

\begin{lemma}\label{lemma Ae-Ve}
$\Ae$ is a $\Ve\text{-superalgebra}$ on a countable set of odd
generators.
\end{lemma}

\begin{proof}
By definition, $\Ae$ is metabelian and can be generated by the
elements $w_1,y_1,y_2,\dots$. Moreover, it is not hard to check
that the only nonzero associators on the basis elements of $\Ae$
are the following:
\begin{align*}
\ass{y_i}{w_i}{y_i}&=\left(y_i\cdot w_i\right)\cdot y_i=a_i\cdot y_i=\varepsilon w_{i+1},\\
\ass{y_i}{y_i}{w_i}&=-y_i\cdot\left(y_i\cdot w_i\right)=-y_i\cdot a_i=-w_{i+1},\\
\ass{y_{i+1}}{a_i}{y_i}&=-y_{i+1}\cdot\left(a_i\cdot
y_i\right)=-y_{i+1}\cdot\left(\varepsilon w_{i+1}\right)=
-\varepsilon a_{i+1},\\
\ass{y_{i+1}}{y_i}{a_i}&=-y_{i+1}\cdot\left(y_i\cdot
a_i\right)=-y_{i+1}\cdot w_{i+1}=-a_{i+1}.
\end{align*}
Thus one can see that the superization of~\eqref{eq EpsilonSymm}
holds in $\Ae$. Finally, it remains to prove that $\Ae$ satisfies
the superization of~\eqref{eq EpsilonNil2}. Actually it suffices
to notice that\linebreak $\left<M_0,U\right>_\varepsilon=0$.
\end{proof}

\subsection{Infiniteness of the basic superrank of $\Ve$}%
\label{SubSec:InfinitenessOfTheBasicSuperrankOfVe}

Let
$f_{k,n}=f_{k,n}\left(u,v,x_1,z_1,x_2,z_2,\dots,x_{kn-1},z_{kn-1},x_{kn}\right)$
be the polynomial
\[
f_{k,n}=\left(uv\right) \sum_{\rho\in{\mathrm R}_{\tau d}}
\sum_{\sigma\in{\mathrm C}_{\rho\tau d}}
{(-1)}^{\left|\sigma\right|}L_{x_{\sigma\rho(1)}}L_{z_1}L_{x_{\sigma\rho(2)}}L_{z_2}\dots
L_{x_{\sigma\rho(kn-1)}}L_{z_{kn-1}}L_{x_{\sigma\rho(kn)}},
\]
where
\[
\tau d= \lefteqn{ \phantom{ \vline\hspace{5pt}
\begin{matrix}
\hline
1      \\
\hline
n+1     \\
\hline
\vdots  \\
\hline
(k-1)n+1 \\
\hline
\end{matrix}
\;\vline }\vline} \lefteqn{ \phantom{ \vline\hspace{5pt}
\begin{matrix}
\hline
1      & 2 \\
\hline
n+1      & n+2  \\
\hline
\vdots  & \vdots  \\
\hline
(k-1)n+1      & (k-1)n+2 \\
\hline
\end{matrix}
\;\vline }\vline} \lefteqn{ \phantom{ \vline\hspace{5pt}
\begin{matrix}
\hline
1      & 2 & \dots  \\
\hline
n+1      & n+2 & \dots  \\
\hline
\vdots  & \vdots  & \ddots  \\
\hline
(k-1)n+1      & (k-1)n+2  & \dots  \\
\hline
\end{matrix}
\;\vline }\vline} \vline\hspace{5pt}
\begin{matrix}
\hline
1      & 2 & \dots  & n \\
\hline
n+1      & n+2 & \dots  & 2n \\
\hline
\vdots  & \vdots  & \ddots & \vdots \\
\hline
(k-1)n+1      & (k-1)n+2  & \dots  &  kn\\
\hline
\end{matrix}
\hspace{5pt}\vline \:.
\]
By Lemma~\ref{lemma Endpsi}, the superpolynomial $\Tilde{f}_{k,n}$
is a superidentity of $\mathfrak M_{r,s}$ for $k=rs+r+1$ and
$n=s+1$. Therefore in view of Lemma~\ref{lemma Ae-Ve}, to prove
the strictness of inclusions $\Ve_{r,s}\subset\Tilde{\mathcal
V}^{\left<\varepsilon\right>}$ it suffices to verify that
$\Tilde{f}_{k,n}$ takes a nonzero value in $\Ae$. Indeed, let us
denote by
\[
\lambda_{k,n}=\lambda_{k,n}\left(L_{x_1},L_{z_1},L_{x_2},L_{z_2},\dots,L_{x_{kn-1}},L_{z_{kn-1}},L_{x_{kn}}\right)
\]
the superpolynomial on operators of left multiplication such that
$\Tilde{f}_{k,n}=(uv)\lambda_{k,n}$. Then by the substitution
$u=y_1$, $v=w_1$, $x_i=y_i$, and $z_i=y_{i+1}$ for $i=1,\dots,kn$,
we have
\[
\Tilde{f}_{k,n}
=a_1\lambda_{k,n}\left(L_{y_1},L_{y_2},L_{y_2},L_{y_3},\dots,L_{y_{kn-1}},L_{y_{kn}},L_{y_{kn}}\right)
=w_2 L^2_{y_2}L^2_{y_3}\dots L^2_{y_{kn}}=w_{kn+1}\neq0.
\]

\smallskip

Theorem~\ref{theorem VeInfin} is proved.

\section{Open problems}%
\label{Sec:OpenProblems}

\begin{enumerate}
\item
Is it true that for every pare of nonnegative integers
$(r,s)\neq (0,0)$
there is a variety
$\mathcal V$ of associative algebras that has the unique basic superrank~$(r,s)$?
\item
What condition should satisfy a set
of
$n\geqslant 2$
pares of nonnegative integers to be the basic spectrum of some variety of algebras?
\item
Do the varieties
$\mathrm{Alt}$, $\mathrm{Jord}$, $\mathrm{Malc}$
have finite basic superranks?
\item
Does every solvable subvariety of
$\mathrm{Alt}$, $\mathrm{Jord}$, $\mathrm{Malc}$
 have a finite basic superrank?
\item
Are there any subvarieties of
$\mathrm{Alt}$, $\mathrm{Jord}$, $\mathrm{Malc}$
of infinite basic superrank?
\item
Does every Spechtian variety of algebras have a finite basic superrank?
\end{enumerate}

\subsection*{Acknowledgments}%
The article was carried out at the Department of Mathematics and Statistics of the University of S\~ao Paulo (IME-USP), Brazil.
The authors are very thankful to the IME-USP for the kind hospitality and the creative atmosphere.
For the first author, this paper is a part of the postdoc project supported by the FAPESP 2010/51880--2 (2011--2014)
and the PNPD/CAPES/UFRN--PPgMAE (2015).
The second author is supported by the FAPESP 2014/09310--5 and the CNPq 303916/2014--1.

\subsection*{Affiliations}%
Alexey Kuz'min\\
PPgMAE\\
Universidade Federal do Rio Grande do Norte\\
Departamento de Mathem\'atica\\
Centro de Ci\^encias Exatas e da Terra\\
Campus Universit\'ario, Lagoa Nova, Natal\\
RN, 59078-970, Brazil\\
e-mail: amkuzmin@ya.ru\\

\smallskip

\noindent Ivan Shestakov\\
Instituto de Mathem\'atica e Estat\'istica\\
Universidade de S\~ao Paulo\\
Rua do Mat\~ao, 1010\\
Cidade Universit\'aria, S\~ao Paulo\\
SP, 05508-090, Brazil\\
e-mail: shestak@ime.usp.br
\end{document}